\documentclass[10pt,a4paper]{article}
\usepackage{amsmath}
\usepackage{amsfonts}
\usepackage{amssymb}
\usepackage{mathdots}
\usepackage{fixmath}

\usepackage{xcolor}
\usepackage{graphicx}
\usepackage{subfigure}

\usepackage{hyperref}
\hypersetup{colorlinks}  
\usepackage{cleveref}
\crefname{assumption}{assumption}{assumptions}
\crefname{figure}{figure}{figures}
\crefname{equation}{}{}

\usepackage[shortlabels]{enumitem}
\setlist[enumerate,1]{label=(\arabic*),ref=\arabic*} 
\setlist[enumerate,2]{label=(\alph*),ref=\theenumi(\alph*)}
\setlist[enumerate,3]{label=\roman*.,ref=\theenumii\roman*}

\usepackage{mathtools}

\DeclarePairedDelimiterXPP\set[1]{}{\{}{\}}{}{

#1}
\DeclarePairedDelimiter\abs{|}{|}
\DeclarePairedDelimiter\N{\|}{\|}
\DeclarePairedDelimiter\inner{\langle}{\rangle}

\def\ol{\overline}
\def\wtd{\widetilde}
\def\what{\widehat}
\def\ii{\iota} 
\def\ee{\mathrm{e}}
\DeclareMathOperator{\diff}{d\!}
\DeclareMathOperator{\diag}{diag}
\DeclareMathOperator{\rank}{rank}
\DeclareMathOperator{\sign}{sign}
\DeclareMathOperator{\HH}{H}
\DeclareMathOperator{\T}{T}
\DeclareMathOperator{\OO}{O}

\usepackage{amsthm}
\newtheorem{theorem}{Theorem}[section]
\newtheorem{lemma}{Lemma}[section]
\theoremstyle{definition}
\newtheorem{remark}{Remark}[section]

\numberwithin{equation}{section}
\numberwithin{figure}{section}
\numberwithin{table}{section}
\allowdisplaybreaks
\def\calI{\mathcal{I}}
\def\calG{\mathcal{G}}
\def\calS{\mathcal{S}}
\usepackage{amsbsy}
\newcommand\bs[1]{\boldsymbol{ #1}}
\def\ba{\boldsymbol{a}}
\def\bx{\boldsymbol{x}}
\def\bk{\boldsymbol{k}}
\def\be{\boldsymbol{e}}
\def\bh{\boldsymbol{h}}
\def\by{\boldsymbol{y}}
\def\bz{\boldsymbol{z}}
\DeclareMathOperator{\inertia}{\mathit{p}}
\DeclareMathOperator{\nullspace}{\mathcal{N}}
\DeclareMathOperator{\nullbasis}{\mathrm{N}}
\DeclareMathOperator{\range}{\mathcal{R}}
\DeclareMathOperator{\domain}{\mathcal{D}}
\DeclareMathOperator{\brillouin}{\mathcal{B}}
\DeclareMathOperator{\neighbor}{\mathcal{L}}
\def\no{\abs}
\def\coo{\inner}

\begin{document}
\title{
Bifurcation Analysis of the Eigenstructure of the Discrete Single-curl Operator in Three-dimensional Maxwell's Equations with Pasteur Media
}

\author{
Xin Liang\thanks{%
Yau Mathematical Sciences Center,
Tsinghua University,
Beijing 100084, China.
E-mail: {\tt liangxinslm@tsinghua.edu.cn}.
Supported in part by the MOST.
}
\and
Zhen-Chen Guo\thanks{%
Department of Mathematics,
Nanjing University,
Nanjing 210093, Jiangsu, China.
E-mail: {\tt guozhenchen@nju.edu.cn}.
Supported in part by the MOST.
}
\and
Tsung-Ming Huang\thanks{%
Department of Mathematics,
National Taiwan Normal University,
Taipei 116, Taiwan.
E-mail: {\tt  min@ntnu.edu.tw}.
Supported in part by the Ministry of Science and Technology (MOST) 108-2115-M-003-012-MY2 and the National Center for Theoretical Sciences (NCTS) in Taiwan.
}
\and
Tiexiang Li\thanks{%
School of Mathematics, Southeast University, Nanjing 211189, Jiangsu, China. 
E-mail: {\tt txli@seu.edu.cn}.
Supported by the National Natural Science Foundation of China (NSFC) 11971105 and the Shing-Tung Yau Center of Southeast University.
}
\and
Wen-Wei Lin\thanks{%
Department of Applied Mathematics,
National Chiao Tung University,
Hsinchu 300, Taiwan.
E-mail: {\tt wwlin@am.nctu.edu.tw}.
Supported in part by the MOST 106-2628-M-009-004-, the NCTS,  and the ST Yau Centre at the National Chiao Tung University.
}
}
\date{\today}
\maketitle

\begin{abstract}
	This paper focuses on studying the bifurcation analysis of the eigenstructure of the $\gamma$-parameterized generalized eigenvalue problem ($\gamma$-GEP) arising in three-dimensional (3D) source-free Maxwell's equations with Pasteur media, where $\gamma$ is the magnetoelectric chirality parameter. For the weakly coupled case, namely, $\gamma < \gamma_{*} \equiv$ critical value, the $\gamma$-GEP is positive definite, which has been well-studied by Chern et.\ al, 2015. For the strongly coupled case, namely, $\gamma > \gamma_{*}$, the $\gamma$-GEP is no longer positive definite, introducing a totally different and complicated structure. For the critical strongly coupled case, numerical computations for electromagnetic fields have been presented by Huang et.\ al, 2018. In this paper, we build several theoretical results on the eigenstructure behavior of the $\gamma$-GEPs. We prove that the $\gamma$-GEP is regular for any $\gamma > 0$, and the $\gamma$-GEP has $2 \times 2$ Jordan blocks of infinite eigenvalues at the critical value $\gamma_{*}$.
Then, we show that the $2 \times 2$ Jordan block will split into a complex conjugate eigenvalue pair that rapidly goes down and up and then collides at some real point near the origin.
	 Next, it will bifurcate into two real eigenvalues, with one moving toward the left and the other to the right along the real axis as $\gamma$ increases.
	  A newly formed state whose energy is smaller than the ground state can be created as $\gamma$ is larger than the critical value.
	   This stunning feature of the physical phenomenon would be very helpful in practical applications. Therefore, the purpose of this paper is to clarify the corresponding theoretical eigenstructure of 3D Maxwell's equations with Pasteur media.
	
\end{abstract}


\smallskip
{\bf Key words.} Bifurcation analysis, Eigenstructure, Maxwell's equations, Pasteur media, Jordan block, Regular matrix pair.

\smallskip
{\bf AMS subject classifications.} Primary: 15A18, 15A22; Secondary  65F15.

\section{Introduction}\label{sec:introduction}

The eigenstructure of the discrete single-curl operator $\nabla\times$ is fundamental and vital for efficient numerical simulations of complex materials.
Here, complex materials, or physically, complex media, imply coupling effects between electric and magnetic fields.
Bianisotropic material is an important class of complex media (see, e.g., \cite[Section~5.3]{wela:2003}),
of which the coupling effects between electric and magnetic fields can be described by the Tellegen representation of the constitutive relations
\[
	\bs B =\bs \mu \bs H+\bs \zeta\bs E,\qquad \bs D=\bs \varepsilon \bs E+\bs \xi \bs H,
\]
where $\bs E,\bs H,\bs D,\bs B$ are the electric, the magnetic fields, the dielectric displacement, and the magnetic induction at the position $\bx$, respectively,
$\bs \mu$ is the permeability, $\bs \varepsilon$ is the permittivity, and $\bs \zeta,\bs \xi$ are magnetoelectric parameters.
Usually, $\bs \mu,\bs \varepsilon,\bs \zeta,\bs \xi$ are dyadics (a.k.a. second-order tensors) of dimension three.
In particular, a bianisotropic medium is also called a biisotropic medium, if $\bs \mu,\bs \varepsilon,\bs \zeta,\bs \xi$ are scalar dyadics, or equivalently,
\[
	\bs\mu = \mu\bs I, \quad
	\bs\varepsilon = \varepsilon\bs I, \quad
	\bs\zeta = \zeta\bs I, \quad
	\bs\xi = \xi\bs I,
\]
where $\bs I$ represents the identity dyadics.
Specifically,
a Pasteur medium (a.k.a. the reciprocal chiral medium) is a type of biisotropic media, where
\begin{equation}
	\xi=\ii\gamma,\qquad
	\zeta=-\ii\gamma,\qquad
	\gamma\ge0. \label{eq:Pasteur}
\end{equation}

Mathematically, the propagation of electromagnetic waves in bianisotropic media is modeled by the three-dimensional (3D) frequency domain source-free Maxwell's equations with the constitutive relations
\begin{align*}
	\nabla\times \bs E = \hphantom{-}\ii \omega \bs B, \qquad& \nabla\cdot \bs B=0,\\
	\nabla\times \bs H = -\ii \omega \bs D, \qquad& \nabla\cdot \bs D=0,
\end{align*}
or equivalently,
\begin{equation*}
	\begin{bmatrix}
		\nabla\times & 0\\
		0 &\nabla\times \\
	\end{bmatrix}
	\begin{bmatrix}
		\bs E \\ \bs H\\
	\end{bmatrix}
	=
	\ii\omega
	\begin{bmatrix}
		\bs \zeta & \bs \mu\\
		-\bs \varepsilon & -\bs \xi\\
	\end{bmatrix}
	\begin{bmatrix}
		\bs E \\ \bs H\\
	\end{bmatrix}
	,\qquad
	\begin{bmatrix}
		\nabla\cdot{} & 0\\
		0 &\nabla\cdot{} \\
	\end{bmatrix}
	\begin{bmatrix}
		\bs D \\ \bs B \\
	\end{bmatrix}
	=0,
\end{equation*}
where $\omega$ is the frequency.
The Bloch theorem, from the theorem named after F.~Bloch (see, e.g., \cite[p. 167]{kitt:2005}), implies that the solutions of the Schr\"odinger equation for a periodic potential must be of a quasi-periodic form, stating that

\begin{quote}
	The eigenfunctions of the wave equation for a periodic potential are the product of a plane wave $\exp(\ii 2\pi \bk\cdot\bx)$ times a function $u_{\bk}(\bx)$ with the periodicity of the crystal lattice.
\end{quote}
Based on the Bloch theorem, the Bloch eigenvectors $\bs E$ and $\bs H$ on any crystal lattice, satisfying the quasi-periodic conditions 
\[
	\bs E(\bx+\ba_\ell)=\bs E(\bx)\exp(\ii2\pi \bk\cdot\ba_\ell),\qquad
	\bs H(\bx+\ba_\ell)=\bs H(\bx)\exp(\ii2\pi \bk\cdot\ba_\ell)
\]
are of interest, where $2\pi\bk$ is the Bloch wave vector in the first Brillouin zone $\brillouin$, and $\ba_\ell,\ell=1,2,3$ are the lattice translation vectors (see, e.g., \cite[p.~34]{jojw:2008}).

Using Yee's scheme \cite{yee:1966}, a finite difference discretization that satisfies the source-free conditions and the quasi-periodicity conditions naturally, the discretized Maxwell's equations are
\begin{equation}\label{eq:maxwell-eq:discrete}
	\begin{bmatrix}
			C & 0\\
			0 & C^{\HH}\\
		\end{bmatrix}
		\begin{bmatrix}
			\be \\ \bh\\
		\end{bmatrix}
		=
		\ii\omega
		\begin{bmatrix}
			\zeta_d & \mu_d\\
			-\varepsilon_d & -\xi_d\\
		\end{bmatrix}
		\begin{bmatrix}
			\be \\ \bh\\
		\end{bmatrix}
		,
\end{equation}
where $\mu_d$, $\varepsilon_d$, $\xi_d$ and $\zeta_d$ are diagonal matrices, and  $C$ is special structured, facilitating the introduction of the fast Fourier transform (FFT) to accelerate numerical simulations \cite{chhh:2013,huhl:2013,hllllt:2018} (see \cref{eq:part_mu_eps_xi_zeta}-\cref{eq:J23} below, for details).

For the Pasteur media, the matrix pair in \cref{eq:maxwell-eq:discrete} is positive definite when the parameter $\gamma$ in \cref{eq:Pasteur} is small, but it becomes an indefinite pair as $\gamma$ becomes larger (see below).
The weakly coupled case, namely, the case in which the matrix pair is positive definite, has been analyzed by Chern et al.~\cite{chhh:2013} in 2015. 
For the strongly coupled case, the matrix pair is no longer a positive-definite matrix pair, introducing a totally different and complicated structure. For the critical strongly coupled case, numerical computations for the electromagnetic fields $\bs E$ and $\bs H$ have been studied by Huang et.\ al.~\cite{hulcl:2019}, but lack of theory makes it difficult to guarantee that the numerical results are valid and reliable.

In this paper, we build several theoretical results on the eigenstructure behavior of the discrete single-curl operator in 3D Maxwell's equations for Pasteur media:
\begin{enumerate}
	\item[(a)] The matrix pair in \cref{eq:maxwell-eq:discrete} is always regular regardless of how large $\gamma$ is;
	\item[(b)] The matrix pair  \cref{eq:maxwell-eq:discrete} has $2 \times 2$ Jordan blocks of the infinite eigenvalues at the critical value $\gamma = \gamma_{*}$. Then, the $2\times 2$ Jordan block will split into a pair of complex conjugate eigenvalues that move rapidly down and up and collide at some real point near the origin to form an associated $2 \times 2$ Jordan block of a real eigenvalue;
	\item[(c)] This $2 \times 2$ Jordan block  will bifurcate into two real eigenvalues such that one moves toward the left and the other to the right along the real axis;
	\item[(d)] A newly formed state whose energy is smaller than the ground state can be created as $\gamma$ is larger than the critical value $\gamma_*$.
\end{enumerate}
The feature exhibited by the physical phenomenon derived from the above three points (b)-(d) is an astonishing finding. 
This discovery would be very useful in practical applications. However, the corresponding theoretical eigenstructure behavior should first be clarified.

{\bf Notation.}
$\ii=\sqrt{-1}$ is the imaginary unit; $\ee=\exp(1)$ is Euler's number. 
For any $n\in\mathbb{N}$, $\eta_n=\ee^{\frac{ 2\pi\ii}{n}}$ is an $n$th root of unity.
For any index set $\calI$, $I^{\calI}$ denotes the diagonal matrix whose $i$th diagonal entry is $1$ for all $i\in\calI$ and $0$ otherwise;
$I^{\calI}_\sigma$ denotes the matrix consisting of all nonzero columns of $I^{\calI}$;
and $\no{\calI}$ denotes the number of its elements.
$I_n$ is the identity matrix of size $n$;
in particular, $I_0$ is an empty matrix;
$e_i$ is the $i$th column of $I_n$.
For a matrix $X$, $X^{\T}$ and $X^{\HH}$ are its transpose and conjugate transpose, respectively;
$\nullspace(X)=\set{v: Xv=0}$ is the kernel of $X$.
For matrices $X,Y$, $X\otimes Y$ is their Kronecker product;
$X\succeq Y$ means that $X-Y$ is positive semidefinite, similarly for ``$\succ$'', ``$\preceq$'', and ``$\prec$''. 
For $m\in \mathbb{N}$, $\alpha\in \mathbb{C}$ and $X\in\mathbb{C}^{n\times n}$, write
\[
\addtolength\arraycolsep{-3pt}
V_{m\times n}(\alpha) :=
\begin{bmatrix}
	1 & 1 & \cdots & 1 \\
	\alpha & \alpha^2 & \cdots & \alpha^n \\
	\vdots & \vdots &  & \vdots \\
	\alpha^{m-1} & \alpha^{2(m-1)} & \cdots & \alpha^{n(m-1)} \\
\end{bmatrix}, \quad
K_m(X) :=
\begin{bmatrix}
	  & I_n  & & \\
	&  & \ddots & \\
	& &   & I_n  \\
	X & & &  \\
\end{bmatrix}_{mn\times mn}.
\addtolength\arraycolsep{3pt}
\]
In particular, write $V_m(\alpha)= V_{m\times 1}(\alpha)$,
$D_m(\alpha)=\alpha\diag( V_m(\alpha))$.

\section{Preliminaries}\label{sec:preliminary}
\subsection{Discretization}\label{ssec:discretization}
It is well known from crystallography that crystal structures can be classified into 14 Bravais lattices \cite{wikipedia1,lattices}. Because of various lattices, the discretized single-curl operators $C$ and $C^{\HH}$ in \cref{eq:maxwell-eq:discrete} on the electric and magnetic fields, respectively, may have different forms. 
In the discretization process,
$\domain_i$ and $\domain_o$ denote the sets including the indices of all vertices inside and outside, respectively, the medium
(usually the domain would be of vacuum or air but could be of another medium).
Then, $\domain=\domain_i\cup\domain_o$ is the discretization grid.
Moreover, $n_1,n_2,n_3$ denote the numbers of grid vertices in the $x,y,z$ directions, respectively,
and $\delta_1,\delta_2,\delta_3$ for the associated mesh lengths. 
Write $n=n_1n_2n_3$.

The discretized Maxwell's equations are \cref{eq:maxwell-eq:discrete},
in which $\zeta_d$, $\mu_d$, $\varepsilon_d$, $\xi_d$ are decided by the shape of the medium, and $C$ is given by Yee's scheme. 
For the former, $\zeta_d$, $\mu_d$, $\varepsilon_d$, $\xi_d$ may not have the same value in three directions of some boundary points. However, for convenience of notation,
in this paper, we consider the case that 
\begin{subequations} \label{eq:part_mu_eps_xi_zeta}
\begin{align}
	\mu_d &= \mu_{io} I_3\otimes I,
	\qquad
	\varepsilon_d = I_3\otimes [\varepsilon_o I^{(o)}+\varepsilon_i I^{(i)}], \label{eq:part_mu_eps}
	\\
	\zeta_d &= -\ii\gamma I_3\otimes I^{(i)},
	\qquad
	\xi_d = \ii\gamma I_3\otimes I^{(i)}, \label{eq:part_xi_zeta}
\end{align}
\end{subequations}
where $\gamma$ is the chirality,
$\varepsilon_i,\varepsilon_o$ are the permittivities inside and outside the medium, respectively, and
$\mu_{io}\equiv 1$ is the permeability. For simplicity, here we denote $I^{(i)} \equiv I^{\domain_i}$ and $I^{(o)} \equiv I^{\domain_o}$.

For the latter, we refer the readers to \cite{hllllt:2018} in order to peruse the details of the whole discretization process.
We provide some basic but important results to ensure this paper is self-contained.
According to the type of the lattice, one of the 14 Bravais lattices, which represent all kinds of crystals, $C$, the discretized single-curl operator on the electric field by Yee's scheme, may have different forms which can be uniformly written as
\begin{equation}\label{eq:C}
	C = 
		\begin{bmatrix}
			0 & -C_3 & C_2 \\
			C_3 & 0 & -C_1 \\
			-C_2 & C_1 & 0 
		\end{bmatrix},
\end{equation}
where 
\begin{subequations}\label{eq:C123}
	\begin{align}
		C_1 & = \delta_1^{-1}[-I_n + I_{n_3}\otimes I_{n_2}\otimes K_{n_1}(\ee^{\ii2\pi\bk\cdot\ba_1})], \\
		C_2 & = \delta_2^{-1}[-I_n + I_{n_3}\otimes K_{n_2}(\ee^{\ii2\pi\bk\cdot\ba_2}J_1)],         \\
		C_3 & = \delta_3^{-1}[-I_n + K_{n_3}(\ee^{\ii2\pi\bk\cdot\ba_3}J_2)],      
	\end{align}
\end{subequations}
with $J_1=J_{1,1}$ and 
\begin{subequations}\label{eq:J23}
	\begin{align}
		J_2 &= \ee^{\ii2\pi\bk\cdot\rho_2\ba_2}
		\begin{bmatrix}
			&\ee^{-\ii2\pi\bk\cdot\ba_2}I_{m_2}\otimes J_{1,3} \\
			I_{n_2-m_2}\otimes J_{1,2}& \\
		\end{bmatrix},
		\\
		J_{1,\ell} &= \ee^{\ii2\pi\bk\cdot\rho_{1,\ell}\ba_1}
		\begin{bmatrix}
			& \ee^{-\ii2\pi\bk\cdot\ba_1} I_{m_{1,\ell}}  \\
			I_{n_1-m_{1,\ell}}  &  \\
		\end{bmatrix}
		,\qquad \ell=1,2,3.
	\end{align}
\end{subequations}
Here,
$\rho_2,\rho_{1,1}\in\set{0,1}$, $\rho_{1,2},\rho_{1,3}\in\set{-1,0,1,2}$ satisfying $\rho_{1,2}-\rho_{1,3}-\rho_{1,1}\in\set{0,1}$,
and $m_2\in[0,n_2]\cap \mathbb{N}$, $m_{1,\ell}\in[0,n_1]\cap \mathbb{N}$ for $\ell=1,2,3$ satisfying $m_{1,2}-m_{1,3}-m_{1,1}\in\set{0,n_1}$.
Write $m_1:=m_{1,1}$, $\rho_1:=\rho_{1,1}$,
and $\what \rho_1:=\rho_2\rho_1+\rho_2\rho_{1,2}+(1-\rho_2)\rho_{1,3}$.

$C_1, C_2, C_3$ are simultaneously diagonalizable by a unitary matrix, which is guaranteed by \cref{lm:C:simul-diag-u}.
\begin{theorem}[{\cite{hllllt:2018}}]\label{lm:C:simul-diag-u}
	$C_1, C_2, C_3$ are simultaneously diagonalizable by the unitary matrix 
	$T=[t_\ell]_{\ell=1,\dots,n}$
	in the forms 
	\begin{subequations}\label{eq:Lambda}
		\begin{align}
			\Lambda_1:=T^{\HH}C_1T&= \delta_1^{-1}[-I_n+\eta_{n_1}^{\bk\cdot\what\ba_1} I_{n_3}\otimes I_{n_2}\otimes D_{n_1}(\eta_{n_1})], \\
			\Lambda_2:=T^{\HH}C_2T&= \delta_2^{-1}[-I_n+\eta_{n_2}^{\bk\cdot\what\ba_2} I_{n_3}\otimes D_{n_2}(\eta_{n_2})\otimes D_{n_1}(\eta_{n_1n_2}^{-m_1})], \\
			\Lambda_3:=T^{\HH}C_3T&= \delta_3^{-1}[-I_n+\eta_{n_3}^{\bk\cdot\what\ba_3} 
			D_{n_3}(\eta_{n_3})
			\otimes D_{n_2}(\eta_{n_3n_2}^{-m_2})
			\otimes D_{n_1}(\eta_{n_1n_3}^{-\what m_1}\eta_{n_1n_2n_3}^{m_1m_2})], 
		\end{align}
	\end{subequations}
	where 
	\begin{equation}\label{eq:T}
		\!\! t_{\coo{i_1,i_2,i_3}}\!=\! \tfrac{1}{\sqrt{n}}
		V_{n_3}(\eta_{n_3}^{\bk\cdot\what\ba_3+i_3}\eta_{n_2n_3}^{-m_2i_2}\eta_{n_1n_3}^{-\what m_1 i_1}\eta_{n_1n_2n_3}^{m_1m_2i_1})\otimes V_{n_2}(\eta_{n_2}^{\bk\cdot\what\ba_2+i_2}\eta_{n_1n_2}^{-m_1i_1})\otimes V_{n_1}(\eta_{n_1}^{\bk\cdot\what\ba_1+i_1}) 
	\end{equation}
	for $i_1,i_2,i_3\in \mathbb{Z}$, $\coo{i_1,i_2,i_3}$ is defined as
\[
	\coo{i_1,i_2,i_3}:=(i_3'-1)n_1n_2+(i_2'-1)n_1+i_1',
\]
where
$i_\ell'=i_\ell+t_\ell n_\ell,
1\le i_\ell'\le n_\ell,
t_\ell\in\mathbb{Z}, \ell=1,2,3$, and
	\begin{align*}
		\what \ba_1&=\ba_1,\\
		\what \ba_2&=\ba_2+(\rho_1-\frac{m_1}{n_1})\what\ba_1,\\
		\what \ba_3&=\ba_3+(\rho_2-\frac{m_2}{n_2})\what \ba_2+[\what \rho_1 -\frac{\what m_1}{n_1}-\rho_2(\rho_1-\frac{m_1}{n_1})]\what\ba_1,\\
		\what m_1&=\rho_2m_1+\rho_2m_{1,2}+(1-\rho_2)m_{1,3}.
	\end{align*}
\end{theorem}
Note that
\begin{enumerate}
	\item[(a)] $\what \ba_1,\what \ba_2,\what \ba_3$ is an orthogonal basis of $\ba_1,\ba_2,\ba_3$,
		and $\N{\what \ba_\ell}_2 = l_\ell = n_\ell\delta_\ell$ for $\ell=1,2,3$.
	\item[(b)] 
		$2\pi\bk\in \mathcal{B}$ implies 
		\begin{equation}\label{eq:k.a}
			b_{l,\ell}\le\bk\cdot\ba_\ell\le b_{u,\ell},\qquad \ell=1,2,3,
		\end{equation}
		where $b_{u,\ell}-b_{l,\ell}\le1$, $b_{l,\ell}\in[-2/3,0], b_{u,\ell}\in[0,5/6]$.
\end{enumerate}

As a result, the singular value decomposition of $C$ can be calculated along the way in \cite{chhh:2013}, which is shown in \cref{lm:C:SVD}. 
\begin{theorem}[{\cite{chhh:2013,hllllt:2018}}]\label{lm:C:SVD}
\addtolength\arraycolsep{-4pt}
	If $\bk\ne0$, then:
	\begin{description}
		\item[\em (a)] $\Lambda_q:=\Lambda_1^{\HH}\Lambda_1+\Lambda_2^{\HH}\Lambda_2+\Lambda_3^{\HH}\Lambda_3\succ0$;
		\item[\em (b)] $\Lambda_p:= \left( 
				\begin{bmatrix}
					0 & -\tau_3 & \tau_2 \\
					\tau_3 & 0 & -\tau_1 \\
					-\tau_2 & \tau_1 & 0 \\
				\end{bmatrix}
			\otimes I_n \right)
			\begin{bmatrix}
				\Lambda_1^{\HH}\\
				\Lambda_2^{\HH}\\
				\Lambda_3^{\HH}\\
			\end{bmatrix}
			$
			is of full column rank, provided $\tau_1\delta_1$, $\tau_2\delta_2$, $\tau_3\delta_3$ are distinct;
		\item[\em (c)] the singular value decomposition (SVD) of $C$ is
			\begin{align}
				C &= (I_3\otimes T)
				\begin{bmatrix}
					-\ol{\Pi_2} & \ol{\Pi_1} & \ol{\Pi_0}
				\end{bmatrix}
				\begin{bmatrix}
					\Lambda_q^{1/2} & & \\
					& \Lambda_q^{1/2} & \\
					& & 0 \\
				\end{bmatrix}
				\begin{bmatrix}
					\Pi_1 & \Pi_2 & \Pi_0
				\end{bmatrix}^{\HH}
				(I_3\otimes T^{\HH}) \nonumber \\
				&\equiv \begin{bmatrix}
				     P_r & P_0
				\end{bmatrix} \begin{bmatrix}
				     \Sigma & 0 \\ 0 & 0
				\end{bmatrix} \begin{bmatrix}
				     Q_r & Q_0
				\end{bmatrix}^{\HH} =P_r \Sigma Q_r^{\HH}, \label{eq:SVD_C}
			\end{align}
			where
			\[
				\Pi_0 = 
				\begin{bmatrix}
					\Lambda_1\\
					\Lambda_2\\
					\Lambda_3\\
				\end{bmatrix}
				\!
				\Lambda_q^{-\frac{1}{2}}
				,
				\Pi_2=\Lambda_p(\Lambda_p^{\HH}\Lambda_p)^{-\frac{1}{2}}
				,
				\ol{\Pi_1}=
				\begin{bmatrix}
					0 & -\Lambda_3 & \Lambda_2 \\
					\Lambda_3 & 0 & -\Lambda_1 \\
					-\Lambda_2 & \Lambda_1 & 0 \\
				\end{bmatrix}
				\!
				\Lambda_p(\Lambda_p^{\HH}\Lambda_p\Lambda_q)^{-\frac{1}{2}}. 
			\]
	\end{description}
\addtolength\arraycolsep{4pt}
\end{theorem}

It is not difficult to observe that there is only one nonzero off-diagonal entry in each column or row of $C_1,C_2,C_3$.
Physically, the entry represents the relation between a mesh node with its surroundings in the mesh grid, or equivalently, the neighbor in the lattice.
The index of such an entry in $C_1,C_2,C_3$ is that of the neighbor of the node along $\ba_1,\ba_2,\ba_3$.
For ease, these $6$ neighbors of a node are called its \emph{lattice neighbors}.
Define
\[
	\neighbor_\ell(i_1,i_2,i_3) :=\set{\coo{i_1',i_2',i_3'}: \text{$C_\ell(i_\ell,i_\ell')\ne0$ or $C_\ell(i_\ell',i_\ell)\ne0$}},\qquad \ell=1,2,3,
\]
and $\neighbor(i_1,i_2,i_3):=\neighbor_1(i_1,i_2,i_3)\cup\neighbor_2(i_1,i_2,i_3)\cup\neighbor_3(i_1,i_2,i_3)$.
Clearly, $\neighbor(i_1,i_2,i_3)$ is the set of the node $\coo{i_1,i_2,i_3}$ and its 6 lattice neighbors.
Furthermore, 
\begin{equation}\label{eq:lattice-neighbor-property}
e_{\coo{i_1,i_2,i_3}}^{\HH}C_{\ell}I^{\domain\setminus\neighbor(i_1,i_2,i_3)}_\sigma=0,\qquad
e_{\coo{i_1,i_2,i_3}}^{\HH}C_{\ell}^{\HH}I^{\domain\setminus\neighbor(i_1,i_2,i_3)}_\sigma=0.
\end{equation}
Moreover, we can define the boundary and interior of an index set $\calI$:
\[
	\partial \calI:=\set{\coo{i_1,i_2,i_3}\in\calI: \neighbor(i_1,i_2,i_3)\setminus\calI\ne\emptyset},\quad
	\calI^{\circ}:=\set{\coo{i_1,i_2,i_3}\in\calI: \neighbor(i_1,i_2,i_3)\subset\calI}.
\]

\subsection{Equivalence of generalized/quadratic eigenvalue problems (GEP /QEP)}\label{ssec:equivalent-eigenvalue-problems}
It is easily seen that the GEP \cref{eq:maxwell-eq:discrete} can be rewritten as 
\begin{multline*}
	\ii
	\begin{bmatrix}
		I_{3n} & 0 \\ -\xi_d\mu_d^{-1} & -I_{3n}
	\end{bmatrix}
	\left( 
	\begin{bmatrix}
		0 & -\ii C \\ \ii C^{\HH} & \ii \xi_d\mu_d^{-1}C-\ii C^{\HH}\mu_d^{-1}\zeta_d
	\end{bmatrix}
	- \omega
	\begin{bmatrix}
		\mu_d & 0 \\ 0 & \varepsilon_d-\xi_d\mu_d^{-1}\zeta_d
	\end{bmatrix}
	\right)
	\\\cdot
	\begin{bmatrix}
		I_{3n} & \mu_d^{-1}\zeta_d \\ 0 & I_{3n}
	\end{bmatrix}
	\begin{bmatrix}
		\bh \\ \be
	\end{bmatrix}
	=0.
\end{multline*}
Together with the choices of $\mu_d,\varepsilon_d,\zeta_d,\xi_d$, we can consider this matrix pair instead
\begin{equation}\label{eq:equivGEP}
	\begin{multlined}
	\addtolength\arraycolsep{-4pt}
		\begin{bmatrix}
			0 & -\ii C \\ \ii C^{\HH} & -\gamma [(I_3\otimes I^{(i)}) C+ C^{\HH}(I_3\otimes I^{(i)})]
		\end{bmatrix}
		- \omega
		\begin{bmatrix}
			I_3\otimes I_n & 0 \\ 0 & I_3\otimes [\varepsilon_o I^{(o)}+(\varepsilon_i-\gamma^2)  I^{(i)}]
		\end{bmatrix}\\
		:=  A_\gamma-\omega B_\gamma
		,
	\addtolength\arraycolsep{4pt}
	\end{multlined}
\end{equation}
also written as a matrix pair $(A_\gamma,B_\gamma)$,
which is equivalent to \cref{eq:maxwell-eq:discrete} in the sense that
\[
	\begin{multlined}
\text{
	$(\omega, 
	\begin{bmatrix}
		\bh-\ii\gamma (I_3\otimes I^{(i)})\be\\\be\\
	\end{bmatrix}
	)$ is an eigenpair of \cref{eq:equivGEP} 
}
\\
	\Leftrightarrow
	\text{
	$(\omega,
	\begin{bmatrix}
		\be\\\bh\\ 
	\end{bmatrix}
	)$ is an eigenpair of \cref{eq:maxwell-eq:discrete}.
}
	\end{multlined}
\]
Moreover, if $(\omega, 
\begin{bmatrix}
	\be\\\bh\\ 
\end{bmatrix}
)$ is an eigenpair of \cref{eq:maxwell-eq:discrete} with $\omega\ne0$,
then $\bh=\ii(\gamma I_3\otimes I^{(i)}-\omega^{-1}C)\be$.
Note that $A_\gamma,B_\gamma$ are Hermitian; $(A_\gamma,B_\gamma)$ is regular if $\gamma\ne \gamma_{*} \equiv \sqrt{\varepsilon_i}$;
$B_\gamma\succ0$ if $\gamma<\gamma_{*}$; $B_\gamma$ is indefinite if $\gamma>\gamma_{*}$.
Thus, all eigenvalues of $(A_\gamma, B_\gamma)$ are real if $\gamma<\gamma_{*}$.

The matrix pair $(A_\gamma, B_\gamma)$ is also equivalent to a Hermitian quadratic matrix polynomial (a Hermitian QEP)
\begin{equation}\label{eq:equivQEP}
	Q_\gamma(\omega):=
		C^{\HH}C
		-\omega\gamma[(I_3\otimes I^{(i)}) C+ C^{\HH}(I_3\otimes I^{(i)})]
		-\omega^2I_3\otimes [\varepsilon_o I^{(o)}+(\varepsilon_i-\gamma^2)  I^{(i)}]
		,
\end{equation}
in the sense that for $\be\ne0$,
\[
\text{
	$(\omega, \be)$ is an eigenpair of \cref{eq:equivQEP}
	$\Leftrightarrow$
	$(\omega,
	\begin{bmatrix}
		\be\\\bh\\
	\end{bmatrix}
	)$ is an eigenpair of \cref{eq:maxwell-eq:discrete}.
}
\]
Clearly, the eigenvalues of $Q_\gamma(\cdot)$ are real or appear in conjugate pairs if nonreal.

Suppose that $\omega$ is an eigenvalue of $Q_\gamma(\cdot)$ and $\be$ is  its corresponding eigenvector.
Then, $Q_\gamma(\omega)\be=0$ gives
\begin{equation}\label{eq:rayleigh-quotient-eq}
	\be^{\HH}Q_\gamma(\omega)\be=
	c(\be)-\omega \gamma b(\be) -\omega^2 [\varepsilon_oa_o(\be)+(\varepsilon_i-\gamma^2)a_i(\be)]
		=0,
\end{equation}
where
\begin{align*}
	c(\be)&:=\be^{\HH}C^{\HH}C \be\ge0,\\
	b(\be)&:=\be^{\HH}[(I_3\otimes I^{(i)}) C+ C^{\HH}(I_3\otimes I^{(i)})]\be=2\Re[\be^{\HH}(I_3\otimes I^{(i)}) C \be]\in \mathbb{R},\\
	a_o(\be)&:=\be^{\HH}[I_3\otimes I^{(o)}]\be\ge0,\\
	a_i(\be)&:=\be^{\HH}[I_3\otimes  I^{(i)}]\be\ge0.
\end{align*}
Furthermore,
\begin{equation}\label{eq:abc=0}
	c(\be)=0 \Leftrightarrow C \be=0,\quad
		a_o(\be)=0 \Leftrightarrow (I_3\otimes I^{(o)}) \be=0,\quad
		a_i(\be)=0 \Leftrightarrow (I_3\otimes I^{(i)}) \be=0.
\end{equation}
By \cref{eq:rayleigh-quotient-eq}, $\omega$ is one of the roots of the scalar function
\begin{subequations} \label{eq:rayleigh-quotient}
\begin{equation}
	\omega_{\pm}(\be)=\frac{\gamma b(\be)\pm \Delta(\be)^{1/2}}{-2[\varepsilon_oa_o(\be)+(\varepsilon_i-\gamma^2)a_i(\be)]}, 
\end{equation}
where
\begin{align}
	\Delta(\be)=\gamma^2b(\be)^2+4c(\be)[\varepsilon_oa_o(\be)+(\varepsilon_i-\gamma^2)a_i(\be)].
\end{align}
\end{subequations}

For $\gamma < \gamma_{*} = \sqrt{\varepsilon_i}$, \cref{eq:maxwell-eq:discrete} and \cref{eq:SVD_C} show that the matrix pair $(A_{\gamma}, B_{\gamma})$ has $2n$ semisimple zero eigenvalues. Furthermore, if $C \be \neq 0$, i.e., $c(\be) >0$, then from \cref{eq:rayleigh-quotient} it follows that $(A_{\gamma}, B_{\gamma})$ has $2n$ positive and $2n$ negative eigenvalues, respectively. Moreover, for $\gamma > \gamma_{*}$, if $\Im\omega\ne0$, then $\bar{\omega}$ is the conjugate eigenvalue.

\subsection{Null-space free GEP}

Since the $6n \times 6n$ Hermitian matrix $A_{\gamma}$ in \cref{eq:equivGEP} has an extensive null space with nullity $2n$, from a computational viewpoint, this would affect and slow down the convergence of the desired smallest positive eigenvalues, and consequently a more compact form for the deflation of all zeros is necessary to be proposed. Fortunately, a $4n \times 4n$ null-space free GEP (NFGEP) has been derived in \cite{chhh:2013}.

\begin{theorem}[\cite{chhh:2013}] \label{thm:NullspaceFree}
     If $\gamma \neq \gamma_{*} \equiv \sqrt{\varepsilon_i}$, then the GEP in \cref{eq:maxwell-eq:discrete} can be reduced to a $4n \times 4n$ NFGEP
\begin{subequations} \label{eq:NFGEP}
\begin{equation}\label{eq3.16}
\widehat{A}_r\by_r=\omega\left(\ii\begin{bmatrix}0 &\Sigma_r^{-1}\\-\Sigma_r^{-1}&0\end{bmatrix}\right)\by_r\equiv \omega \widehat{B}_r\by_r,
\end{equation}
and
\[
      \begin{bmatrix}
           \bh \\ \be\\
      \end{bmatrix} = \ii \begin{bmatrix} -I_{3n}&-\zeta_{d}   \\
     \xi_{d} &\varepsilon_{d}
    \end{bmatrix}^{-1} {\diag} \left( P_{r}, Q_{r} \right) \by_r,
\]
where
\begin{equation}\label{eq3.17}
\widehat{A}_r:=\widehat{A}_r(\gamma)\equiv \diag(P_r^{\HH},Q_r^{\HH})\begin{bmatrix}\zeta_d &-I_{3n}\\ I_{3n}&0\end{bmatrix}\begin{bmatrix}\Phi^{-1} &0\\0&I_{3n}\end{bmatrix}\begin{bmatrix}\xi_d &I_{3n}\\ -I_{3n}&0\end{bmatrix}\diag(P_r,Q_r)
\end{equation}
\end{subequations}
with $\Phi:=\Phi(\gamma)\equiv I_3 \otimes \left[ \varepsilon_o I^{(o)} + (\varepsilon_i -\gamma^2) I^{(i)} \right]$ by \cref{eq:part_mu_eps_xi_zeta}.     
\end{theorem}

\begin{theorem} \label{thm:derivative_omega}
     For $\gamma \lesssim  \gamma_{*} \equiv \sqrt{\varepsilon_i}$, it holds generally that
     \[
     \begin{cases}
           \frac{d \omega(\gamma)}{d \gamma} \geq 0, & \mbox{if } \omega(\gamma) > 0, \\
           \frac{d \omega(\gamma)}{d \gamma} \leq 0, & \mbox{if } \omega(\gamma) < 0,
     \end{cases}
     \]
     i.e., all positive and negative eigenvalues of $(\widehat{A}_{r}, \widehat{B}_r)$ either move toward the right and the left, respectively, or stop motionless as $\gamma$ becomes close to $\gamma_{\ast}$.
\end{theorem}
\begin{proof}
     For $\gamma < \gamma_*$, $\widehat{A}_r$ in \cref{eq3.17} is positive definite.  There is an eigenvector $\by_r(\gamma)$ with $\by_r^{\HH}(\gamma) \widehat{A}_r(\gamma) \by_r(\gamma) = 1$ such that
     \begin{subequations}
     \begin{equation}
          \frac{1}{\omega(\gamma)} = \by_r^{\HH}(\gamma) \widehat{B}_r \by_r(\gamma). \label{eq:RaRe_Br}
     \end{equation}
     Because of $\by_r^{\HH}(\gamma) \widehat{A}_r(\gamma) \by_r(\gamma) = 1$, we have
     \begin{equation}
		 2 \Re [(\by_r^{\HH}(\gamma))^{\prime} \widehat{A}_r(\gamma) \by_r(\gamma)] + \by_r^{\HH}(\gamma) \widehat{A}_{r}^{\prime}(\gamma) \by_r(\gamma) = 0. \label{eq:der_RaRe_Ar}
     \end{equation}
     \end{subequations}
     Taking the derivative of $\gamma$ in \cref{eq:RaRe_Br}, we have
     \begin{alignat*}{2}
		 -\frac{\omega^{\prime}(\gamma)}{\omega(\gamma)^2} &= 2 \Re[(\by_r^{\HH}(\gamma))^{\prime} \widehat{B}_r \by_r(\gamma)] & \\*
		 &= \frac{2}{\omega(\gamma)}\Re [(\by_r^{\HH}(\gamma))^{\prime} \widehat{A}_r(\gamma) \by_r(\gamma)] &  \quad & (\mbox{by }  \cref{eq3.16}) \\*
           &= -\frac{1}{\omega(\gamma)} \by_r^{\HH}(\gamma) \widehat{A}_{r}^{\prime}(\gamma) \by_r(\gamma) &\quad &  (\mbox{by \cref{eq:der_RaRe_Ar}}) \\*
           &= -\frac{1}{\omega(\gamma)}\bz_r^{\HH}(\gamma) \left( I_3 \otimes \begin{bmatrix}
                 \frac{2\gamma \varepsilon_i}{(\varepsilon_i - \gamma^2)^2}  I^{(i)} & -\frac{\ii (\varepsilon_i + \gamma^2)}{(\varepsilon_i - \gamma^2)^2} I^{(i)} \\ 
                 \frac{\ii (\varepsilon_i + \gamma^2)}{(\varepsilon_i - \gamma^2)^2}  I^{(i)} & \frac{2\gamma}{(\varepsilon_i - \gamma^2)^2}  I^{(i)}
           \end{bmatrix} \right) \bz_r(\gamma) &\quad &(\mbox{by \cref{eq3.17}}) \\*
           &\equiv -\frac{1}{\omega(\gamma)} d(\gamma), &
     \end{alignat*}
     where $\bz_r (\gamma) := \diag(P_r, Q_r) \by_r(\gamma)$. Since the matrix 
     $$W = \frac{1}{(\varepsilon_i - \gamma^2)^{2}} \begin{bmatrix} 
           2 \gamma \varepsilon_i & -\ii (\varepsilon_i + \gamma^2) \\ \ii (\varepsilon_i + \gamma^2) & 2\gamma
     \end{bmatrix}$$ 
     is orthogonally congruent to 
     $\begin{bmatrix} 
           4 \gamma (\varepsilon_i+1) (\varepsilon_i - \gamma^2)^{-2} & 0 \\
           0 & - [4 \gamma (\varepsilon_i + 1)]^{-1}
     \end{bmatrix}$, it holds generically that $d(\gamma) \geq 0$ as $\gamma \nearrow \gamma_{\ast}$. This implies
     that $\omega^{\prime}(\gamma)$ has the same sign as $\omega(\gamma)$, for $d(\gamma) > 0$, and $\omega^{\prime}(\gamma) = 0$ for $d(\gamma) =0$. 
\end{proof}

\section{Eigenstructure of the discrete single-curl operator}\label{sec:eigen-structure-of-discrete-single-curl-operator}
\subsection{Regularity}\label{ssec:regularity}
Clearly, if $\gamma\ne \gamma_{*} = \sqrt{\varepsilon_i}$, since $B_\gamma$ is nonsingular, the matrix $(A_\gamma, B_\gamma)$ is regular.
In the following, we will provide a condition to make $(A_\gamma,B_\gamma)$ regular at $\gamma=\gamma_{*}$. 
For ease, in this subsection, we write $A=A_{\gamma_{*}}$, $B=B_{\gamma_{*}}$.

First, we locate the nullspace. It is easy to see that 
$\nullspace(B)=\begin{bmatrix}
	0 \\ I_3\otimes  I^{(i)} \\
\end{bmatrix}$.
By \cref{lm:C:SVD}, from the SVD of $C$, we know that $\nullspace(C)=\range((I_3\otimes T)\Pi_0)$ and $\nullspace(C^{\HH})=\range((I_3\otimes T)\ol{\Pi_0})
$.
Thus, $ \nullspace(A_\gamma)=\range(L_\gamma) $, where
\[
	L_\gamma=
	\begin{bmatrix}
		-\ii\gamma(I_3\otimes I^{(i)})(I_3\otimes T)\Pi_0&(I_3\otimes T)\ol{\Pi_0}\\
		(I_3\otimes T)\Pi_0&0\\
	\end{bmatrix}
	.
\]
Any column of $L_\gamma$ is an eigenvector corresponding to the eigenvalue $0$ of either the matrix $A_\gamma$ or the matrix pair $(A_\gamma,B_\gamma)$.
In particular, for $(A_\gamma,B_\gamma)$, we call these eigenvalues \emph{trivial zero eigenvalues}.

Then, we try to find an equivalent condition such that $(A,B)$ is regular. Hereafter, we use the notations $I_{\sigma}^{(i)}$ and $I_{\sigma}^{(o)}$ to denote the matrices consisting of the nonzero columns of $I^{(i)}$ and $I^{(o)}$, respectively.

\begin{theorem}\label{lm:regularity-equiv-cond}
	For $\bz_\ell\in\nullspace(\Lambda_\ell^{\HH}),\ell=1,2,3$ satisfying
	\begin{equation}\label{eq:proper:z}
			I^{(o)}T\bz_1=I^{(o)}T\bz_2=I^{(o)}T\bz_3,
	\end{equation}
	$\calS(\bz_1, \bz_2, \bz_3)$ denotes the set of all nonzero $\bx_1$ that satisfies
	\begin{subequations}\label{eq:constraint:x1}
		\begin{gather}
			\label{eq:constraint:x1:1}
			(I_n+\delta_1\Lambda_1) \bx_1 - \bz_1
			=(I_n+\delta_2\Lambda_2) \bx_1 - \bz_2
			=(I_n+\delta_3\Lambda_3) \bx_1 - \bz_3
			, \\
			\label{eq:constraint:x1:2}
		I^{(o)}T\Lambda_1 \bx_1=0.
		\end{gather}
	\end{subequations}
	Then, $(A,B)$ is regular if and only if $\calS(\bz_1, \bz_2, \bz_3)=\set{0}$ for any proper $\bz_\ell$.
\end{theorem}
\begin{proof}
	Suppose that $\bx\in\nullspace(A)\cap\nullspace(B)$.
	Since $\nullspace(B)=\begin{bmatrix}
		0 \\ I_3\otimes  I^{(i)}
	\end{bmatrix}$,
	$\bx$ must be of the form
	$
	\begin{bmatrix}
		0 \\ (I_3\otimes  I^{(i)}_\sigma)\be_0
	\end{bmatrix}
	$ with some suitable vector $\be_0$.
	Additionally, with some vectors $\be_1$ and $\be_2$, $\bx$ must be of the form
	\begin{equation}
		L_{\gamma_{*}}
		\begin{bmatrix}
			\be_1 \\ \ii\gamma_{*}\be_2\\
		\end{bmatrix}
		=
		\begin{bmatrix}
			-\ii\gamma_{*}(I_3\otimes I^{(i)})(I_3\otimes T)\Pi_0 \be_1+\ii\gamma_{*}(I_3\otimes T)\ol{\Pi_0}\be_2\\
			(I_3\otimes T)\Pi_0 \be_1\\
		\end{bmatrix}
		=
		\begin{bmatrix}
			0 \\ (I_3\otimes  I^{(i)}_\sigma)\be_0
		\end{bmatrix}. \label{eq:pf_regular_1}
	\end{equation}
	From the second equation in \cref{eq:pf_regular_1}, the first equation implies that
	\[
	     (I_3 \otimes I^{(i)}) (I_3 \otimes T) \Pi_0 \be_1 = (I_3 \otimes I^{(i)}) (I_3 \otimes I_{\sigma}^{(i)}) \be_0 
	     = (I_3 \otimes I_{\sigma}^{(i)}) \be_0 = (I_3 \otimes T) \ol{\Pi_0} \be_2, 
	\]
	which yields that
	\begin{equation}
		(I_3\otimes  [T^{\HH}I^{(i)}_\sigma])\be_0= \Pi_0 \be_1 = \ol{\Pi_0}\be_2, \label{eq:pf_regular_3}
	\end{equation}
	with $\be_0\ne0$, and $\be_1$, $\be_2$ being not simultaneously zero. Let $\bx_1=\Lambda_q^{-1/2}\be_1$, $\bx_2=\Lambda_q^{-1/2}\be_2$ and $\be_0=
	\begin{bmatrix}
		\by_1^{\T} & \by_2^{\T} & \by_3^{\T}
	\end{bmatrix}^{\T}$. The equations in \cref{eq:pf_regular_3}
	are equivalent to 
	\begin{equation}\label{eq:constraint:x1x2}
		T^{\HH} I^{(i)}_\sigma \by_1 = 	\Lambda_1 \bx_1 = \Lambda_1^{\HH} \bx_2,\ \ 
		T^{\HH} I^{(i)}_\sigma \by_2 = 	\Lambda_2 \bx_1 = \Lambda_2^{\HH} \bx_2,\ \
		T^{\HH} I^{(i)}_\sigma \by_3 = 	\Lambda_3 \bx_1 = \Lambda_3^{\HH} \bx_2.
	\end{equation} 
	Thus,
	\[
		\by_1=(I^{(i)}_\sigma)^{\HH}T\Lambda_1 \bx_1,\quad
		\by_2=(I^{(i)}_\sigma)^{\HH}T\Lambda_2 \bx_1,\quad
		\by_3=(I^{(i)}_\sigma)^{\HH}T\Lambda_3 \bx_1.
	\]
	Noticing that $\delta_i^{-1}(\ee^{\ii\theta}-1)=-\delta_i^{-1}(\ee^{-\ii\theta}-1)\ee^{\ii\theta}$ for any $\theta\in \mathbb{R}$,
	we have $\Lambda_{\ell}=-\Lambda_\ell^{\HH}(I_n+\delta_{\ell}\Lambda_{\ell})$.
	By \cref{eq:constraint:x1x2}, we have
	\begin{equation} \label{eq:pf_regular_4}
		-\bx_2 = (I_n+\delta_1\Lambda_1) \bx_1 -\bz_1
		=(I_n+\delta_2\Lambda_2) \bx_1 - \bz_2
		=(I_n+\delta_3\Lambda_3) \bx_1 - \bz_3,
	\end{equation}
	namely, \cref{eq:constraint:x1:1},
	where $\bz_\ell\in\nullspace(\Lambda_\ell^{\HH})$.
	Left-multiplying $I^{(o)}T$ on the sides of \cref{eq:constraint:x1x2} and noticing that $I^{(o)}T T^{\HH}I^{(i)}_\sigma=0$, we have
	\[
		I^{(o)}T\Lambda_1 \bx_1=0,\quad
		I^{(o)}T\Lambda_2 \bx_1=0,\quad
		I^{(o)}T\Lambda_3 \bx_1=0,
	\]
	which is equivalent to \cref{eq:constraint:x1:2} by the proper condition \cref{eq:proper:z}.
	Therefore,
	\[
		\nullspace(A)\cap\nullspace(B)
		= \set*{
			\begin{bmatrix}
				0 \\
				(I_3 \otimes T) \Pi_0\Lambda_q^{1/2} \bx_1
			\end{bmatrix}
			: 
			\bx_1\in\calS
		}.
	\]
	Noticing that 
	\[
		\begin{bmatrix}
			0 \\
			\Pi_0\Lambda_q^{1/2} \bx_1
		\end{bmatrix}
		\ne0 \
		\Leftrightarrow \
		\bx_1^{\HH}\Lambda_q^{1/2}\Pi_0^{\HH}\Pi_0\Lambda_q^{1/2} \bx_1= \bx_1^{\HH}\Lambda_q \bx_1\ne 0  \
		\Leftrightarrow \ \bx_1\ne 0.
	\] 
	We have shown that $\calS(\bz_1, \bz_2, \bz_3)=\set{0}$ if and only if $\nullspace(A)\cap\nullspace(B) = \set{0}$. From Theorem 4.1 in \cite{dzli:1991}, it follows that the regularity of $(A, B)$ is equivalent to $\calS(\bz_1, \bz_2, \bz_3)=\set{0}$. 
\end{proof}

At this point, we have an equivalence condition that $(A,B)$ is regular. In practice, because the linear system in \cref{eq:constraint:x1} is overdetermined, the condition $\calS(\bz_1, \bz_2, \bz_3)=\set{0}$ is generically held, and therefore, $A - \omega B$ is always regular. A very lengthy and complex proof for showing the regularity of $A - \omega B$ can be found in the Appendix.
From Theorem 4.1 of \cite{dzli:1991}, it is possible that $(A, B)$ has a defective infinite eigenvalue with a Jordan block of at most $2$. Below we will give a very loose condition that $(A, B)$ has this type of eigenvalue.

\begin{theorem}\label{lm:Jordan-block-equiv-cond}
	Suppose that $(A,B)$ is regular.
	The matrix pair $(A,B)$ has a defective infinite eigenvalue associated with a Jordan block of size two if and only if
	there exist $\bx_1, \bx_2, \bx_3$, not all zero vectors, such that
	\begin{subequations}\label{eq:lm:Jordan-block-equiv-cond}
		\begin{align}
			(I^{(i)}_\sigma)^{\HH}M_2I^{(i)}_\sigma \bx_3
			-(I^{(i)}_\sigma)^{\HH}M_3I^{(i)}_\sigma \bx_2
			&=0,\\
			(I^{(i)}_\sigma)^{\HH}M_3I^{(i)}_\sigma \bx_1
			-(I^{(i)}_\sigma)^{\HH}M_1I^{(i)}_\sigma \bx_3
			&=0,\\
			(I^{(i)}_\sigma)^{\HH}M_1I^{(i)}_\sigma \bx_2
			-(I^{(i)}_\sigma)^{\HH}M_2I^{(i)}_\sigma \bx_1
			&=0,
		\end{align}
	\end{subequations}
	where $M_\ell=C_\ell-C_\ell^{\HH}$.
\end{theorem}
\begin{proof}
	Clearly, $(A,B)$ has a defective infinite eigenvalue if and only if there exist nonzero vectors $\bx, \by$ such that
	\[
		B \bx=0, \quad B\by = A \bx.
	\]
	Thus, $\bx \in\nullspace(B),  A \bx\in\range(B)$. Since $(A,B)$ is regular, $A \bx\ne0$.
	Noticing $\nullspace(B)=\range(B)^\perp$, the equations are equivalent to 
	\[
		0\ne \bx \in\nullspace(B), \quad \nullspace(B)^{\HH}A \bx =0,
	\]
	namely,
	\[
		\nullspace\left( \nullbasis(B)^{\HH}A\nullbasis(B) \right)\ne 0,
	\]
	where $\nullbasis(B)$ is the basis matrix of $\nullspace(B)$.
	Note that with $\gamma_{*} = \sqrt{\varepsilon_i}$, we have
	\begin{align*}
		\nullbasis(B)^{\HH}A\nullbasis(B)
		&= \begin{bmatrix}
			0 \\ I_3\otimes I^{(i)}_\sigma\\
		\end{bmatrix}^{\HH}
		\begin{bmatrix}
			0 & -\ii C \\ \ii C^{\HH} & -\gamma_{*}[(I_3\otimes I^{(i)})C+C^{\HH}(I_3\otimes I^{(i)})]\\
		\end{bmatrix}
		\begin{bmatrix}
			0 \\ I_3\otimes I^{(i)}_\sigma\\
		\end{bmatrix}
		\\&=-\gamma_{*} (I_3\otimes I^{(i)}_\sigma)^{\HH}[(I_3\otimes I^{(i)})C+C^{\HH}(I_3\otimes I^{(i)})](I_3\otimes I^{(i)}_\sigma)
		\\&=-\gamma_{*} (I_3\otimes I^{(i)}_\sigma)^{\HH}[C+C^{\HH}](I_3\otimes I^{(i)}_\sigma)
		\\&=-\gamma_{*} (I_3\otimes I^{(i)}_\sigma)^{\HH}\begin{bmatrix}
			0 & -C_3+C_3^{\HH} & C_2-C_2^{\HH} \\
			C_3-C_3^{\HH} & 0 & -C_1+C_1^{\HH} \\
			-C_2+C_2^{\HH} & C_1-C_1^{\HH} & 0 \\
		\end{bmatrix}(I_3\otimes I^{(i)}_\sigma)
		\\&=-\gamma_{*} (I_3\otimes I^{(i)}_\sigma)^{\HH}\begin{bmatrix}
	 		0 & -M_3 & M_2 \\
	 		M_3 & 0 & -M_1 \\
	 		-M_2 & M_1 & 0 \\
		\end{bmatrix}(I_3\otimes I^{(i)}_\sigma)
		.
	\end{align*}
	Clearly, $\begin{bmatrix}
		\bx_1^{\T} & \bx_2^{\T} & \bx_3^{\T}
	\end{bmatrix}^{\T}\in \nullspace\left(\nullbasis(B)^{\HH}A\nullbasis(B)\right)$ is equivalent to \cref{eq:lm:Jordan-block-equiv-cond}. Finally, from Theorem 4.1 of \cite{dzli:1991}, the defective infinite eigenvalue has a Jordan block of size two. 
\end{proof}

\begin{theorem}\label{lm:Jordan-block}
	Suppose that $(A,B)$ is regular and $n_\ell>2$.
	The matrix pair $(A,B)$ has a defective infinite eigenvalue, as long as
	a mesh node with its $6$ lattice neighbors (see \cref{ssec:discretization}) are inside the medium,
	i.e.,
	there exist some $i_\ell\in[1,n_\ell],\ell=1,2,3$, such that $\neighbor(i_1,i_2,i_3)\subset \domain_i$, or equivalently, $\coo{i_1,i_2,i_3}\in\domain_i^{\circ}$.

	As a result, $(A,B)$ has a defective infinite eigenvalue, as long as $n_\ell$ is large enough. 
\end{theorem}
\begin{proof}
	First, we claim that under the assumption there exists a nonzero $\by$ such that
	\begin{equation}\label{eq:constraint:y}
		(I^{(o)}_\sigma)^{\HH}M_{\ell}I^{(i)}_\sigma \by=0,\quad \ell=1,2,3; \qquad
		\sum_{\ell=1}^3 \by^{\HH}(I^{(i)}_\sigma)^{\HH}M_{\ell}^{\HH}M_{\ell}I^{(i)}_\sigma \by\ne 0 .
	\end{equation}
	Then, let $\bx_{\ell}=(I^{(i)}_\sigma)^{\HH}M_{\ell}I^{(i)}_\sigma \by$.
	Note that $M_{\ell}=-M_{\ell}^{\HH}$ and $M_{\ell}=T(\Lambda_{\ell}-\Lambda_{\ell}^{\HH})T^{\HH}$.
	We know $M_1,M_2,M_3$ are simultaneously diagonalizable by the unitary matrix $T$.
	Thus, by the orthogonality relations in the first equations of \cref{eq:constraint:y}, it holds that
	\begin{align*}
		&\quad (I^{(i)}_\sigma)^{\HH}M_{\ell}I^{(i)}_\sigma \bx_{\ell'}
		-(I^{(i)}_\sigma)^{\HH}M_{\ell'}I^{(i)}_\sigma \bx_{\ell} \\*
		&= (I^{(i)}_\sigma)^{\HH}M_{\ell}I^{(i)}_\sigma (I^{(i)}_\sigma)^{\HH}M_{\ell'}I^{(i)}_\sigma \by
		-(I^{(i)}_\sigma)^{\HH}M_{\ell'}I^{(i)}_\sigma (I^{(i)}_\sigma)^{\HH}M_{\ell}I^{(i)}_\sigma \by \\*
		&= (I^{(i)}_\sigma)^{\HH}M_{\ell}[I-I^{(o)}_\sigma (I^{(o)}_\sigma)^{\HH}]M_{\ell'}I^{(i)}_\sigma \by
		-(I^{(i)}_\sigma)^{\HH}M_{\ell'}[I-I^{(o)}_\sigma (I^{(o)}_\sigma)^{\HH}]M_{\ell}I^{(i)}_\sigma \by \\*
		&= (I^{(i)}_\sigma)^{\HH}[M_{\ell}M_{\ell'}-M_{\ell'}M_{\ell}]I^{(i)}_\sigma \by 
		= 0
		,
	\end{align*}
	or equivalently, $\bx_\ell$ satisfy \cref{eq:lm:Jordan-block-equiv-cond}.
	On the other hand,
	\begin{align*}
		\bx_1^{\HH} \bx_1+ \bx_2^{\HH} \bx_2+ \bx_3^{\HH} \bx_3
		&= \sum_\ell \by^{\HH}(I^{(i)}_\sigma)^{\HH}M_{\ell}^{\HH}I^{(i)}_\sigma (I^{(i)}_\sigma)^{\HH}M_{\ell}I^{(i)}_\sigma \by
		\\&= \sum_\ell \by^{\HH}(I^{(i)}_\sigma)^{\HH}M_{\ell}^{\HH}[I-I^{(o)}_\sigma (I^{(o)}_\sigma)^{\HH}]M_{\ell}I^{(i)}_\sigma \by
		\\&= \sum_\ell \by^{\HH}(I^{(i)}_\sigma)^{\HH}M_{\ell}^{\HH}M_{\ell}I^{(i)}_\sigma \by
		\ne0
		,
	\end{align*}
	which means that $\bx_\ell$ are not all zeros.
	Therefore, by \cref{lm:Jordan-block-equiv-cond}, we have the result.

	Finally we prove the claim.
	Since $\neighbor(i_1,i_2,i_3)\subset \domain_i$, we know $\domain_o\subset \domain\setminus\neighbor(i_1,i_2,i_3)$.
	By \cref{eq:lattice-neighbor-property},
	\[
		e_{\coo{i_1,i_2,i_3}}^{\HH}M_\ell I^{(o)}_\sigma=
		e_{\coo{i_1,i_2,i_3}}^{\HH}M_\ell^{\HH} I^{(o)}_\sigma=0.
	\]
	On the other hand,
	$M_\ell e_{\coo{i_1,i_2,i_3}}$ is a column of $M_\ell$ and thus nonzero, as long as $n_\ell>2$.
	Note that there exists $\by$ such that $e_{\coo{i_1,i_2,i_3}}=I^{(i)}_\sigma \by$
	because $e_{\coo{i_1,i_2,i_3}}\in\range(I^{(i)}_\sigma)$.
	Clearly, this $\by$ satisfies \cref{eq:constraint:y}. 
\end{proof}

\begin{remark}\label{rk:lm:Jordan-block}
	Write $M=\begin{bmatrix}
		M_1^{\T} & M_2^{\T} & M_3^{\T}
	\end{bmatrix}^{\T}$.
	From the proofs of \cref{lm:Jordan-block-equiv-cond,lm:Jordan-block},
	we can see that $\begin{bmatrix}
		0 \\ (I_3\otimes [I^{(i)}_\sigma (I^{(i)}_\sigma)^{\HH}])Me_{\coo{i_1,i_2,i_3}}\\
	\end{bmatrix}=\begin{bmatrix}
		0 \\ Me_{\coo{i_1,i_2,i_3}}\\
	\end{bmatrix}
	$ is a corresponding eigenvector of the defective infinite eigenvalue.
\end{remark}

\subsection{The eigenvalue behavior when \texorpdfstring{$\gamma\to\gamma_{*}+0$}{}}\label{ssec:the-eigenvalue-behavior-when-gammatovarepsilon_i-1-2-0-}
First, we observe the eigenvalues of $(A, B) = (A_{\gamma_{*}},B_{\gamma_{*}})$.
Write
\[
	G_m	:= 
	\begin{bmatrix}
		&         & 1       & 0 \\
		& \iddots & \iddots &   \\
		1 & \iddots &         &   \\
		0 &         &         &   \\
	\end{bmatrix}
	_{m\times m}
	,\qquad
	F_m :=
	\begin{bmatrix}
		&         &         & 1 \\
		&         & \iddots &   \\
		& \iddots &         &   \\
		1 &         &         &   \\
	\end{bmatrix}
	_{m\times m}.
\]
By \cite[Theorem~5.10.1]{golr:2005} (also \cite[Theorem~6.1]{laro:2005}),
any Hermitian regular matrix pair $(A,B)$ is congruent to a Hermitian matrix pair which is a direct sum of the following types of blocks:
\newcommand\typ[2][B]{\textbf{#1-#2}}
\begin{description}
\item[\typ{c}.] $ (\begin{bmatrix} &\beta_jF_{k_j}+G_{k_j}\\ \ol{\beta_j}F_{k_j}+G_{k_j}&\end{bmatrix},F_{2k_j}),\quad \beta_j\in \mathbb{C}\setminus \mathbb{R}, \quad j=1, \dots, n_c$, with possible replacement of $\beta_j$ by $\ol{\beta_j}$;
	\item[\typ{r}.] $ \mu_j(\alpha_jF_{k_j}+G_{k_j},F_{k_j}), \quad\alpha_j\in \mathbb{R}, \mu_j\in\set{1,-1}, \quad j=1, \dots, n_r$;
	\item[\typ{$\infty$}.] $ \nu_j(F_{k_j},G_{k_j}), \quad \nu_j\in\set{1,-1}, \quad j=1, \dots, n_{\infty}$.
\end{description}
The form is uniquely determined by $(A,B)$ up to a combination of permutations of those blocks.
Furthermore, $\beta_j,\ol{\beta_j}$ are finite nonreal eigenvalues of $(A,B)$; $\alpha_j$ is its real eigenvalue. 
Those $k_j$'s corresponding to the same value $\alpha$ are called the \emph{partial multiplicities} of $\alpha$;
all $\mu_j$'s and $\nu_j$'s are called the \emph{sign characteristic} of $(A,B)$.
A real eigenvalue $\alpha_j$ with the corresponding $\mu_j=1$ ($\mu_j=-1$) is called an eigenvalue of \emph{positive type} ({\em negative type}).

First, as a consequence of the regularity and \cref{lm:Jordan-block}, we have \cref{lm:eigenvalue-infty}.
\begin{theorem}\label{lm:eigenvalue-infty}
	The matrix pair $(A, B)$ has at most $6\no{\domain_i}$ infinite eigenvalues, each of which is either semisimple or of positive type 
	and associated with a Jordan block of size $2$,
	and at least $6n-6\no{\domain_i}$ semisimple eigenvalues of positive type.
\end{theorem}
\begin{proof}
	Note that $B \succeq0$ with $\dim\nullspace(B)=3\no{\domain_i}$ and $(A, B)$ is regular.
	The result is a direct consequence of \cite[Theorem 4.1]{dzli:1991} and \cref{lm:Jordan-block}. 
\end{proof}

Then, we provide a necessary condition of the existence of nonreal eigenvalues, or equivalently,
a necessary condition that $Q_\gamma( \omega )$ has a nonreal eigenvalue.
\begin{theorem}\label{lm:necessary-cond-of-nonreal-eigenvalue}
	For $\gamma\to\gamma_{*}+0$, there exist purely imaginary eigenvalues. 
	If $(\omega,\be)$ is an eigenpair of $Q_\gamma(\cdot)$ with $\Im \omega\ne0$,
	then:
	\begin{description}
		\item[\em (a)] $(I_3\otimes I^{(o)})\be=0$, $(I_3\otimes I^{(i)})\be\ne0$;
		\item[\em (b)] $C \be\ne0$, $\Re[\be^{\HH}(I_3\otimes I^{(i)})C \be]=0$;
		\item[\em (c)] $\omega$ is pure imaginary, and $\omega=\pm(\gamma^2-\varepsilon_i)^{-1/2}\frac{\N{C \be}_2}{\N{\be}_2}\ii$.
		\item[\em (d)] $\abs{\omega}$ becomes smaller as $\gamma$ becomes larger.
	\end{description}
\end{theorem}
\begin{proof}
        By \cref{lm:Jordan-block,lm:eigenvalue-infty}, $(A_{\gamma_{*}}, B_{\gamma_{*}})$ has a $2 \times 2$ Jordan block $W_{\gamma_{*}}(\lambda)  \equiv \begin{bmatrix} 0 & 1 \\ 1 & 0 \end{bmatrix} - \lambda \begin{bmatrix} 1 & 0 \\ 0 & 0 \end{bmatrix}$ at infinity. Let $W_{\gamma}(\lambda) \equiv \begin{bmatrix} 0 & 1 \\ 1 & 0 \end{bmatrix} - \lambda \begin{bmatrix} 1 & \eta \\ \eta & -\eta \end{bmatrix}$ be a small perturbation of $W_{\gamma_{*}}(\lambda)$ with $\eta \to 0^{+}$, as $\gamma \to \gamma_{*}^{+}$. Then, $W_{\gamma}(\lambda)$ has a complex eigenvalue $\omega$ of the form $\omega \equiv \frac{1}{1+\eta} + \ii \frac{1}{\sqrt{\eta}(1+\eta)}$ with $\Im \omega \neq 0$.
	By \cref{eq:rayleigh-quotient},
	it implies that $\Delta(\be)<0$.
	Then, with $\gamma\to\gamma_{*}^{+}$, it forces that
	\[
		b(\be)=0, \quad c(\be)a_o(\be)=0, \quad c(\be)a_i(\be)>0,
	\]
	which, together with \cref{eq:abc=0}, implies the results of (a), (b) and (c). From \cref{eq:rayleigh-quotient-eq}, 
	item~(d) holds because
	\[
		\frac{\diff \,(\ii\omega)}{\diff \gamma}
		=\frac{2\gamma\omega^2 a_i(\be)}{\pm \abs{\Delta(\be)}^{1/2}}
		\qquad \Rightarrow \qquad
		\frac{\diff \,\abs{\omega}}{\diff \gamma} 
		=\frac{2\gamma\omega^2 a_i(\be)}{\abs{\Delta(\be)}^{1/2}}
		<0.
	\] 
\end{proof}

\subsection{Behavior of real eigenvalues}\label{ssec:behavior-of-real-eigenvalues}
As we pointed out above, all the eigenvalues of the matrix pair $(A_\gamma,B_\gamma)$ are real if $\gamma< \gamma_* \equiv \sqrt{\varepsilon_i}$.
Now we begin to check the case $\gamma>\gamma_*$.

First, we build a relation between the change of inertia and the change of the number of real eigenvalues.
For any Hermitian matrix $X$, denote by $\inertia_+(X),\inertia_-(X)$ the positive and negative indices of the inertia of $X$, respectively.

\begin{theorem}\label{lm:inertia-and-eigenvalue}
	Let $C_\gamma(\omega)=A_\gamma-\omega B_\gamma$ with $\omega\in \mathbb{R}$, where $A_\gamma, B_\gamma$ are defined as in \cref{eq:equivGEP},
	and $B_\gamma$ is nonsingular.
	\begin{description}
		\item[\em (a)] 
			if $\inertia_+( C_{\gamma+0}(\omega))-\inertia_+(C_{\gamma-0}(\omega))=t$ and $\inertia_-(C_{\gamma+0}(\omega))-\inertia_-(C_{\gamma-0}(\omega))=-t$,
			then $\omega$ is an eigenvalue associated with $t$ Jordan blocks of odd size, of the matrix pair $(A_\gamma, B_\gamma)$, which is either of positive type and monotonically increasing, or of negative type and monotonically decreasing;
		\item[\em (b)] 
			if $\inertia_+( C_{\gamma+0}(\omega))-\inertia_{+}(C_{\gamma-0}(\omega))=-t$ and $\inertia_-(C_{\gamma+0}(\omega))-\inertia_-(C_{\gamma-0}(\omega))=t$,
			then $\omega$ is an eigenvalue associated with $t$ Jordan blocks of odd size, of the matrix pair $(A_\gamma, B_\gamma)$, which is either of positive type and monotonically decreasing, or of negative type and monotonically increasing.
	\end{description}
\end{theorem}
\begin{proof}
	First, we consider the inertia of the matrix $C_\gamma(\omega)$ for a fixed $\gamma$.
	Let us discuss the inertia of those blocks one after another, except \typ{$\infty$}. Note that from \cref{ssec:equivalent-eigenvalue-problems}, we know that both $\omega$ and $\bar{\omega}$ with $\Im \omega \neq 0$ are eigenvalues of $(A_{\gamma}, B_{\gamma})$. 
	\begin{enumerate}
		\item[(a)] \typ{c}:
		the corresponding matrix is $$L = \begin{bmatrix} &(\beta_j-\omega)F_{m_j}+G_{m_j}\\ (\ol{\beta_j}-\omega)F_{m_j}+G_{m_j}&\end{bmatrix}$$ whose indices of inertia are $\inertia_+(L)=m_j,\inertia_-(L)=m_j$;
		\item[(b)] \typ{re}, \typ{r} of even size $k_j=2s$:
			the corresponding matrix is $$L=\mu_j([\alpha_j-\omega]F_{k_j}+G_{k_j})$$
			whose indices of inertia are $\inertia_+(L)=s,\inertia_-(L)=s$ if $\omega\ne\alpha_j$, or the same as $\mu_jF_{k_j-1}$, namely, $s-1+\frac{1+\mu_j}{2}$ and $s-1+\frac{1-\mu_j}{2}$, if $\omega=\alpha_j$;
		\item[(c)] \typ{ro}, \typ{r} of odd size $k_j=2s-1$:
			the corresponding matrix is $$L=\mu_j([\alpha_j-\omega]F_{k_j}+G_{k_j})$$
			whose indices of inertia are $\inertia_+(L)=s-1+\frac{1+\mu_j\sign(\alpha_j-\omega)}{2},\inertia_-(L)=s-1+\frac{1-\mu_j\sign(\alpha_j-\omega)}{2}$ if $\omega\ne\alpha_j$, or the same as $\mu_jF_{k_j-1}$, namely, $s-1$ and $s-1$ if $\omega=\alpha_j$.
	\end{enumerate}
	Recall the form of $A_\gamma=C_\gamma(0)$. 
	It can be seen that $\inertia_+(A_\gamma)=\inertia_-(A_\gamma)$ for any $\gamma$.
	Thus, counting the inertia of the blocks of different types, we have
	\[
		\text{
			no.\ of \typ{re}($_{\alpha=0}^{\mu=1}$)
			+ no.\ of \typ{ro}($_{\alpha\ne0}^{\mu=1}$)
			=
			no.\ of \typ{re}($_{\alpha\ne0}^{\mu=-1}$)
			+ no.\ of \typ{ro}($_{\alpha=0}^{\mu=-1}$).
		}
	\]

	Note that the eigenvalues, as the functions of the entries of the matrix, are continuous.
	As $\gamma$ goes from $\gamma_1$ to $\gamma_2$, the structure of the blocks may change in one or some combination of the ways below, provided $\omega$ is not an eigenvalue of either $(A_{\gamma_1},B_{\gamma_1})$ or $(A_{\gamma_2},B_{\gamma_2})$:
	\begin{enumerate}
		\item[(a)] $\typ{c} \to \typ{c}$: the indices of inertia are the same;
		\item[(b)] $\typ{re} \to \typ{c}$: the indices of inertia are the same;
		\item[(c)] $\typ{re} \to \typ{re}$: the indices of inertia are the same;
		\item[(d)] $\typ{ro}(\mu=1) \to \typ{ro}(\mu=1)$:
			the indices of inertia are the same if $\omega$ is not between $\alpha(\gamma_1)$ and $\alpha(\gamma_2)$,
			or the positive index decreases $1$ and the negative index increases $1$ if $\alpha(\gamma_2)<\omega<\alpha(\gamma_1)$,
			or the positive index increases $1$ and the negative index decreases $1$ if $\alpha(\gamma_2)>\omega>\alpha(\gamma_1)$;
		\item[(e)] $\typ{ro}(\mu=1) \to \typ{ro}(\mu=-1)$: 
			the indices of inertia are the same if $\omega$ is between $\alpha(\gamma_1)$ and $\alpha(\gamma_2)$,
			or the positive index decreases $1$ and the negative index increases $1$ if $\omega<\alpha(\gamma_1),\omega<\alpha(\gamma_2)$,
			or the positive index increases $1$ and the negative index decreases $1$ if $\omega>\alpha(\gamma_1),\omega>\alpha(\gamma_2)$;
		\item[(f)] $\typ{re} \to \typ{re} + \typ{re}$: the indices of inertia are the same;
		\item[(g)] $\typ{ro}(\mu=1) \to \typ{re}(\mu=1) + \typ{ro}(\mu=1)$:
			the indices of inertia are the same if $\omega$ is not between $\alpha(\gamma_1)$ and $\alpha(\gamma_2)$,
			or the positive index decreases $1$ and the negative index increases $1$ if $\alpha(\gamma_2)<\omega<\alpha(\gamma_1)$,
			or the positive index increases $1$ and the negative index decreases $1$ if $\alpha(\gamma_2)>\omega>\alpha(\gamma_1)$, noticing that $\alpha(\gamma_2)$ is the eigenvalue of \typ{ro};
		\item[(h)] $\typ{re}(\mu=1) \to \typ{ro}(\mu=1) + \typ{ro}(\mu=1)$:
			the indices of inertia are the same if $\omega$ is between $\alpha(\gamma_2)$ and $\alpha'(\gamma_2)$,
			or the positive index decreases $1$ and the negative index increases $1$ if $\omega<\alpha(\gamma_2),\omega<\alpha'(\gamma_2)$,
			or the positive index increases $1$ and the negative index decreases $1$ if $\omega>\alpha(\gamma_2),\omega>\alpha'(\gamma_2)$, noticing that $\alpha(\gamma_2),\alpha'(\gamma_2)$ are the eigenvalues of two \typ{ro}'s;
		\item[(i)] all the reverse (go from right to left) and all the opposite (change $\mu$'s sign).
	\end{enumerate}

	For simplicity, we will not list all the cases.
	Some illustrations are given below.
	\begin{enumerate}
		\item[(I)] As we said before, any change of the structure of the blocks can be expressed as one or some combination of the cases listed above.
			For example, $\typ{c} \to \typ{ro}(\mu=1) + \typ{ro}(\mu=-1)$ can be treated as $\typ{c} \to \typ{re} \to \typ{ro}(\mu=1) + \typ{ro}(\mu=1) \to \typ{ro}(\mu=1) + \typ{ro}(\mu=-1)$, namely, the combination of the reverse of item~(b), item~(h), and item~(e). 
			Note that we use a sequential form to represent it but it does not occur sequentially. However, representing the form sequentially does not affect counting the inertia.
		\item[(II)] Noticing that $p_+(A_\gamma)=p_-(A_\gamma)$, item~(e) or item~(h) cannot happen singly.
			For example, for item~(e), the case in which $\typ{ro}(\mu=1) \to \typ{ro}(\mu=-1)$ happens on only one block and other blocks remain the same will break the equality that  $p_+(A_\gamma)=p_-(A_\gamma)$.
		\item[(III)] Item~(e) may happen together with its reverse, namely, $\typ{ro}(\mu=1) \to \typ{ro}(\mu=-1)$ and $\typ{ro}(\mu=-1) \to \typ{ro}(\mu=1)$ happen simultaneously.
			However, if the involved eigenvalues are not the same, then $\omega$ cannot be both between $\alpha_1(\gamma_1),\alpha_1(\gamma_2)$ and between $\alpha_2(\gamma_1),\alpha_2(\gamma_2)$;
			otherwise, we can treat the case as item~(d) that happens with its opposite, namely, $\typ{ro}(\mu=1) \to \typ{ro}(\mu=1)$ and $\typ{ro}(\mu=-1) \to \typ{ro}(\mu=-1)$ happen simultaneously.
		\item[(IV)] If $\omega$ is between $\alpha(\gamma_1)$ and $\alpha(\gamma_2)$, then according to the continuity, $\omega$ must be an eigenvalue of $(A_{\gamma_{*}},B_{\gamma_{*}})$ for some $\gamma_{*}$ between $\gamma_1,\gamma_2$.
	\end{enumerate}
	After a systematical check, we have the result as the summary. 
\end{proof}

\begin{theorem} \label{thm:bifurcation_ew}
     The case $\mbox{\em \typ{c}} \to \mbox{\em \typ{re}}  \to \mbox{\em \typ{ro}}(\mu=1) + \mbox{\em \typ{ro}}(\mu=-1)$ occurs generically, i.e., if the complex conjugate eigenvalue curves $\beta(\gamma)$ and $\bar{\beta}(\gamma)$ collide at $\alpha(\gamma_1) = \beta(\gamma_1) = \bar{\beta}(\gamma_1) \in \mathbb{R}$ with $\gamma = \gamma_1 > \gamma_{*}$, then it would bifurcate into two real eigenvalues $\alpha_{\ell}(\gamma^{+}_1) ( \mu = 1)$ and $\alpha_{r}(\gamma_1^{+}) (\mu = -1)$.
\end{theorem}
\begin{proof}
      \cref{lm:eigenvalue-infty,lm:necessary-cond-of-nonreal-eigenvalue} show that $(A_{\gamma_{*}}, B_{\gamma_{*}})$ has a $2 \times 2$ Jordan block at infinity and a  purely imaginary eigenpair $\{ \pm \ii \omega(\gamma) \}$ is created for $\gamma \to \gamma_{*}^{+}$ with $\frac{d |\omega|}{d \gamma} < 0$. Then, from \cref{eq:rayleigh-quotient}, we denote the complex conjugate eigenvalue pair by $\{ \beta(\gamma), \bar{\beta}(\gamma) \}$ with $\beta(\gamma_{*}^{+}) = \ii \omega(\gamma_{*}^{+})$ and $\bar{\beta}(\gamma_{*}^{+}) = -\ii \omega(\gamma_{*}^{+})$ which will collide at $\alpha_{\ell}(\gamma_1) = \alpha_{r}(\gamma_1) \in \mathbb{R}$ with $\gamma_1 > \gamma_{*}$. Consequently, 
      it is sufficient to show that the tangent lines of $\beta(\gamma)\cup \bar{\beta}(\gamma)$ and $\alpha(\gamma) \equiv \alpha_{\ell}(\gamma) \cup \alpha_r(\gamma)$ at $\gamma = \gamma_1$ are orthogonal to the real $x$- and imaginary $y$-axes, respectively. Without loss of generality, we consider the following combinations with small perturbation $\eta \equiv \eta(\gamma) \to 0^{+}$ as $\gamma \to \gamma_1^{\pm}$. 
      
      \begin{enumerate}
            \item[(a)] $\typ{c} (\gamma \to \gamma_1^{-})$:
                  \[
                       \left( \begin{bmatrix}
                             0 & \alpha(\gamma_1) + \sqrt{\eta}\ii \\ \alpha(\gamma_1) - \sqrt{\eta}\ii & 0
                       \end{bmatrix}, \begin{bmatrix}
                             0 & 1 \\ 1 & 0
                       \end{bmatrix} \right) \overset{\mbox{\small eq.}}{\sim} \left( \begin{bmatrix}
                             1 & \alpha(\gamma_1) \\ \alpha(\gamma_1)  & -\eta
                       \end{bmatrix}, \begin{bmatrix}
                             0 & 1 \\ 1 & 0
                       \end{bmatrix} \right) 
                  \]
                  Here and hereafter, ``$\overset{\mbox{\small eq.}}{\sim}$'' denotes the equivalence transformation between two matrix pairs.
                  
            \item[(b)]  $\typ{re}(\gamma = \gamma_1)$:
                  \[
                        \left( \begin{bmatrix}
                             1 & \alpha(\gamma_1) \\ \alpha(\gamma_1)  & 0
                       \end{bmatrix}, \begin{bmatrix}
                             0 & 1 \\ 1 & 0
                       \end{bmatrix} \right)  
                  \]
                  
            \item[(c)] $\typ{ro}(\mu=1) + \typ{ro}(\mu=-1)(\gamma \to \gamma_1^{+})$:
                 \[
                       \left( \begin{bmatrix}
                             1 & \alpha(\gamma_1) \\ \alpha(\gamma_1)  & \eta
                       \end{bmatrix}, \begin{bmatrix}
                             0 & 1 \\ 1 & 0
                       \end{bmatrix} \right) \overset{\mbox{\small eq.}}{\sim} \left( \begin{bmatrix}
                             \alpha(\gamma_1) + \sqrt{\eta} & 0  \\ 0 & \alpha(\gamma_1) - \sqrt{\eta} 
                       \end{bmatrix}, \begin{bmatrix}
                             1 & 0 \\ 0 & 1
                       \end{bmatrix} \right) 
                 \]
      \end{enumerate}
      
      Note that here it holds that
      \begin{subequations}
           \begin{align}
                 \beta(\gamma_1^{-}) &\equiv \alpha(\gamma_1) + \sqrt{\eta} \ii, \label{eq:beta_gamma} \\
                 \alpha_{\ell,r}(\gamma_1^{+}) &\equiv \alpha(\gamma_1) \mp \sqrt{\eta}\ (\mu = \pm 1). \label{eq:alpha_lr}
           \end{align}
      \end{subequations}
      In \cref{eq:beta_gamma}, by letting $y = \sqrt{\eta}$, we then have $\left. \frac{d y}{d\eta} \right|_{\eta = 0} = \infty$ if and only if $\left. \frac{d\eta}{d y} \right|_{y=0\ (\gamma = \gamma_1)} = 0$. Similarly, letting $x = \sqrt{\eta}$, from \cref{eq:alpha_lr} it follows that $\left. \frac{d\eta}{d x} \right|_{x=0\ (\gamma = \gamma_1)} = 0$. As a result, we have the theorem. 
\end{proof}

From \cref{eq:SVD_C}, the SVD of $C$ is written as
\begin{align*}
     C = \begin{bmatrix}
				     P_r & P_0
				\end{bmatrix} \begin{bmatrix}
				     \Sigma & 0 \\ 0 & 0
				\end{bmatrix} \begin{bmatrix}
				     Q_r & Q_0
				\end{bmatrix}^{\HH} =P_r \Sigma Q_r^{\HH} \quad \mbox{with } \ \Sigma \succ 0.
\end{align*}
We denote
\begin{equation}
     U_0 := (I_3 \otimes I^{(o)}) Q_0, \quad U_1 := P_r^{\HH} (I_3 \otimes I^{(i)}) Q_0, \quad U_2 := P_0^{\HH} ( I_3 \otimes I^{(i)} ) Q_0 \label{eq:mtx_Ui}
\end{equation}
and use ``$\sim$'' to denote the congruence transformation that two Hermitian matrices have the same inertia. We have the following useful lemma. 

\begin{lemma}\label{lm:nonreal-eigenvalue-burst}
	Suppose $\bk\ne0$ and $n_\ell>4$. Then, it holds that $U_2 \neq 0$ and
	\begin{equation}
	     A_{\gamma} - \alpha B_{\gamma} \sim C_{\gamma}(\alpha) \equiv \diag \left( -\alpha I_{3n}, \frac{1}{\alpha} \Sigma^2, -\alpha \left[ \varepsilon_0 U_0^{\HH} U_0 + \varepsilon_i U_1^{\HH} U_1 + (\varepsilon_i - \gamma^2) U_2^{\HH} U_2 \right]\right) \label{eq:inertia_A_B}
	\end{equation}
	as $\alpha \to 0$. Furthermore, $\alpha = 0^{-}$ and $\alpha = 0^{+}$ are, respectively, the eigenvalues of $(A_{\gamma}, B_{\gamma})$ with some $\gamma \equiv \gamma^{-} > \gamma_{*}$ and $\gamma \equiv \gamma^{+} > \gamma_{*}$. 
\end{lemma}
\begin{proof}
We consider the inertia of $A_\gamma-\alpha B_\gamma$ for a sufficiently small $\alpha$. 
\addtolength\arraycolsep{-5pt}
\begin{align*}
	& A_\gamma-\alpha B_\gamma \\
	=& 
	\begin{bmatrix}
		-\alpha I_{3n}& -\ii C \\ \ii C^{\HH} & -\gamma [(I_3\otimes I^{(i)}) C+ C^{\HH}(I_3\otimes I^{(i)})]-\alpha I_3\otimes [\varepsilon_o I^{(o)}+(\varepsilon_i-\gamma^2)  I^{(i)}]
	\end{bmatrix}
	\\ 
	\sim&
	\begin{bmatrix}
		-\alpha I_{3n}& 0 \\ 0 & -\gamma [(I_3\otimes I^{(i)}) C+ C^{\HH}(I_3\otimes I^{(i)})]-\alpha I_3\otimes [\varepsilon_o I^{(o)}+(\varepsilon_i-\gamma^2)  I^{(i)}] +\frac{1}{\alpha}C^{\HH}C
	\end{bmatrix}
	\\
	:=&
	\begin{bmatrix}
		-\alpha I_{3n}& 0 \\ 0 & \frac{1}{\alpha}C^{\HH}C-D
	\end{bmatrix}
	,
\end{align*}
\addtolength\arraycolsep{5pt}
where $D=\gamma [(I_3\otimes I^{(i)}) C+ C^{\HH}(I_3\otimes I^{(i)})]+\alpha I_3\otimes [\varepsilon_o I^{(o)}+(\varepsilon_i-\gamma^2)  I^{(i)}]$.
Thus, for $\wtd D=\frac{1}{\alpha}C^{\HH}C-D$, we have 
\begin{align*}
	\wtd D
	&\sim
	\begin{bmatrix}
		Q_r^{\HH}\\Q_0^{\HH}
	\end{bmatrix}
	\wtd D
	\begin{bmatrix}
		Q_r&Q_0
	\end{bmatrix}
	=
	\begin{bmatrix}
		Q_r^{\HH}\wtd DQ_r &Q_r^{\HH}\wtd DQ_0\\
		Q_0^{\HH}\wtd DQ_r &Q_0^{\HH}\wtd DQ_0\\
	\end{bmatrix}
	\\&\sim
	\begin{bmatrix}
		Q_r^{\HH}\wtd DQ_r & 0\\
		0 &Q_0^{\HH}\wtd DQ_0-Q_0^{\HH}\wtd DQ_r(Q_r^{\HH}\wtd DQ_r)^{-1}Q_r^{\HH}\wtd DQ_0\\
	\end{bmatrix}
	.
\end{align*}
Let us discuss the terms involved one by one.
\begin{enumerate}
	\item[(a)] The term $Q_r^{\HH}\wtd DQ_r$:
		since
		\[
			Q_r^{\HH}\wtd DQ_r
			= Q_r^{\HH} \left(\frac{1}{\alpha}C^{\HH}C-D\right) Q_r
			= \frac{1}{\alpha} \left( \Sigma^2-\alpha Q_r^{\HH}DQ_r\right)
			,
		\]
		we have
		\begin{align*}
			(Q_r^{\HH}\wtd DQ_r)^{-1}
			&= \alpha \left( \Sigma^2-\alpha Q_r^{\HH}DQ_r \right)^{-1}
			\\&= \alpha \Sigma^{-1} \left( I-\alpha \Sigma^{-1}Q_r^{\HH}DQ_r\Sigma^{-1} \right)^{-1} \Sigma^{-1}
			\\&:= \alpha \Sigma^{-1} \left( I-\alpha D_1 \right)^{-1}\Sigma^{-1}
			\\&= \alpha \Sigma^{-1}\left( I+ \alpha D_1 \left[I-\alpha D_1 \right]^{-1}\right)\Sigma^{-1}
			,
		\end{align*}
		where $D_1 =\Sigma^{-1}Q_r^{\HH}DQ_r\Sigma^{-1}$.
	\item[(b)] The term $Q_0^{\HH}\wtd DQ_0$:
		\begin{align*}
			Q_0^{\HH}\wtd DQ_0
			&= Q_0^{\HH}\left( \frac{1}{\alpha}C^{\HH}C-D \right)Q_0
			= -Q_0^{\HH}DQ_0
			\\&= 
			\begin{multlined}[t]
				-Q_0^{\HH}\Big( \gamma [(I_3\otimes I^{(i)}) C+ C^{\HH}(I_3\otimes I^{(i)})]\\
				+\alpha I_3\otimes [\varepsilon_o I^{(o)}+(\varepsilon_i-\gamma^2)  I^{(i)}] \Big)Q_0
			\end{multlined}
			\\&= -\alpha Q_0^{\HH} I_3\otimes [\varepsilon_o I^{(o)}+(\varepsilon_i-\gamma^2)  I^{(i)}]Q_0
			\\&= -\alpha [\varepsilon_o U_0^{\HH} U_0 + (\varepsilon_i-\gamma^2)(U_1^{\HH} U_1+U_2^{\HH}U_2)  ]
			,
		\end{align*}
		where $U_0$, $U_1$ and $U_2$ are defined in \cref{eq:mtx_Ui}.
	\item[(c)] The term $Q_0^{\HH}\wtd DQ_r$:
		\begin{align*}
			Q_0^{\HH}\wtd DQ_r
			&= Q_0^{\HH}\left( \frac{1}{\alpha}C^{\HH}C-D \right)Q_r
			= -Q_0^{\HH}DQ_r
			\\&= 
			\begin{multlined}[t]
				-Q_0^{\HH}\Big( \gamma [(I_3\otimes I^{(i)}) C+ C^{\HH}(I_3\otimes I^{(i)})]\\
				+\alpha I_3\otimes [\varepsilon_o I^{(o)}+(\varepsilon_i-\gamma^2)  I^{(i)}] \Big)Q_r
			\end{multlined}
			\\&= -\gamma Q_0^{\HH}(I_3\otimes I^{(i)})P_r\Sigma - \alpha Q_0^{\HH}I_3\otimes [\varepsilon_o I^{(o)}+(\varepsilon_i-\gamma^2)  I^{(i)}]Q_r 
			\\&= -\gamma U_1^{\HH}\Sigma - \alpha Q_0^{\HH}I_3\otimes [\varepsilon_o I^{(o)} - (\varepsilon_i-\gamma^2)  (I_n-I^{(o)})]Q_r 
			\\&= -\gamma U_1^{\HH}\Sigma - \alpha (\varepsilon_o - \varepsilon_i+\gamma^2)U_0^{\HH}Q_r - \alpha (\varepsilon_i-\gamma^2)Q_0^{\HH}Q_r
			\\&= -\gamma U_1^{\HH}\Sigma - \alpha (\varepsilon_o - \varepsilon_i+\gamma^2)U_0^{\HH}Q_r
			.
		\end{align*}
\end{enumerate}
Thus,
\begin{align*}
	& Q_0^{\HH}\wtd DQ_0-Q_0^{\HH}\wtd DQ_r (Q_r^{\HH}\wtd DQ_r)^{-1} Q_r^{\HH}\wtd DQ_0
	\\&= 
	\begin{multlined}[t]
		-\alpha [\varepsilon_o U_0^{\HH} U_0 + (\varepsilon_i-\gamma^2)(U_1^{\HH} U_1+U_2^{\HH} U_2)  ] \\
		-[-\alpha (\varepsilon_o - \varepsilon_i+\gamma^2)U_0^{\HH}Q_r-\gamma U_1^{\HH}\Sigma]
		\alpha \Sigma^{-1}\left( I+ \alpha D_1 \left[I-\alpha D_1\right]^{-1}\right)\Sigma^{-1} \\
		\times [-\alpha (\varepsilon_o - \varepsilon_i+\gamma^2)U_0^{\HH}Q_r-\gamma U_1^{\HH}\Sigma]^{\HH}
	\end{multlined}
	\\&=
		-\alpha [\varepsilon_o U_0^{\HH} U_0 + (\varepsilon_i-\gamma^2)(U_1^{\HH} U_1+U_2^{\HH}U_2)  ]
		- (-\gamma U_1^{\HH}\Sigma)
		\alpha \Sigma^{-2}
		(-\gamma U_1^{\HH}\Sigma)^{\HH}
		+\OO(\alpha^2)
	\\&=
		-\alpha [\varepsilon_o U_0^{\HH} U_0 + (\varepsilon_i-\gamma^2)(U_1^{\HH} U_1+U_2^{\HH}U_2)  ]
		-\alpha\gamma^2 U_1^{\HH} U_1
		+\OO(\alpha^2)
	\\&:=
	-\alpha
	\left[ 
	\varepsilon_o U_0^{\HH} U_0 +\varepsilon_i U_1^{\HH} U_1 
	+(\varepsilon_i-\gamma^2) U_2^{\HH} U_2
	\right]
	+\OO(\alpha^2).
\end{align*}
To summarize, for a sufficiently small $\alpha$,
\begin{align*}
	& A_\gamma-\alpha B_\gamma \\*
	\sim &
	\begin{bmatrix}
		-\alpha I_{3n} & & \\
		& \frac{1}{\alpha}\Sigma^2 - Q_r^{\HH}DQ_r & \\
		&& 
	-\alpha
	\left[ 
	\varepsilon_o U_0^{\HH} U_0 +\varepsilon_i U_1^{\HH} U_1
	+(\varepsilon_i-\gamma^2) U_2^{\HH}U_2
	\right]
	+\OO(\alpha^2)
	\end{bmatrix}
	\\*
	\approx&
	\begin{bmatrix}
		-\alpha I_{3n} & & \\
		& \frac{1}{\alpha}\Sigma^2 & \\
		&& 
	-\alpha
	\left[ 
	\varepsilon_o U_0^{\HH}U_0 +\varepsilon_i U_1^{\HH} U_1 
	+(\varepsilon_i-\gamma^2) U_2^{\HH}U_2
	\right]
	\end{bmatrix}
	\qquad \text{as $\alpha \to 0$}
	.
\end{align*}

Now, we prove that $U_2\ne0$.
Note that by \cref{lm:C:SVD}
\[
	Q_0=
	\begin{bmatrix}
		T\Lambda_1 \\ T\Lambda_2 \\ T\Lambda_3
	\end{bmatrix}
	\Lambda_q^{-1/2},
	\qquad
	P_0=
	\begin{bmatrix}
		T\Lambda_1^{\HH} \\ T\Lambda_2^{\HH} \\ T\Lambda_3^{\HH}
	\end{bmatrix}
	\Lambda_q^{-1/2}.
\]
Thus,
\begin{align*}
	U_2
	& =P_0^{\HH}(I_3\otimes I^{(i)})Q_0 \\
	&=\Lambda_q^{-1/2}\left( \Lambda_1 T^{\HH}I^{(i)}T\Lambda_1 +\Lambda_2 T^{\HH}I^{(i)}T\Lambda_2+\Lambda_3 T^{\HH}I^{(i)}T\Lambda_3 \right)\Lambda_q^{-1/2}.
\end{align*}
Consider $\Lambda_q^{1/2}e_{\coo{n_1,m_1,\what m_1}} := \Lambda_q^{1/2}e_j$.
Then,
\begin{align*}
	& (\Lambda_q^{1/2}e_j)^{\HH} U_2\Lambda_q^{1/2}e_j \\
	=& \left( \delta_1^{-2} (\eta_{n_1}^{\bk\cdot\what\ba_1}-1)^2  + \delta_2^{-2} (\eta_{n_2}^{\bk\cdot\what\ba_2}-1)^2 + \delta_3^{-2} (\eta_{n_3}^{\bk\cdot\what\ba_3}-1)^2 \right)e_j^{\HH} T^{\HH}I^{(i)}Te_j
	.
\end{align*}
Since $Te_j=t_j$ of which each entry is nonzero by \cref{lm:C:simul-diag-u}, $t_j^{\HH}I^{(i)}t_j>0$.
Note that
\begin{align*}
	\Re\left( \sum_{\ell=1}^{3} \delta_\ell^{-2} (\eta_{n_\ell}^{\bk\cdot\what\ba_\ell}-1)^2 \right)
	&=\Re\left( \sum_{\ell=1}^{3} \delta_\ell^{-2} (1-2\eta_{n_\ell}^{\bk\cdot\what\ba_\ell}+\eta_{n_\ell}^{2\bk\cdot\what\ba_\ell}) \right)
	\\&= \sum_{\ell=1}^{3} \delta_\ell^{-2} \left( 1-2\cos\frac{2\pi\bk\cdot\what\ba_\ell}{n_\ell}+\cos2\frac{2\pi\bk\cdot\what\ba_\ell}{n_\ell} \right)
	\\&= -2\sum_{\ell=1}^{3} \delta_\ell^{-2}\cos\frac{2\pi\bk\cdot\what\ba_\ell}{n_\ell}(1-\cos\frac{2\pi\bk\cdot\what\ba_\ell}{n_\ell})
	<0
	,
\end{align*}
in which the inequality holds because
by \cref{lm:C:simul-diag-u} and \cref{eq:k.a},
$\abs{\bk\cdot\what\ba_\ell}\le1$,
and when $n_\ell>4$, $\cos\frac{2\pi\bk\cdot\what\ba_\ell}{n_\ell}>0$;
it is impossible that $\cos\frac{2\pi\bk\cdot\what\ba_1}{n_1}=\cos\frac{2\pi\bk\cdot\what\ba_2}{n_2}=\cos\frac{2\pi\bk\cdot\what\ba_3}{n_3}=1$, contradicting $\bk\ne0$.
As a result, $U_2\ne0$.

Finally, \cref{lm:inertia-and-eigenvalue} shows that $\alpha$ becomes an eigenvalue of the matrix pair $(A_{\gamma_*}, B_{\gamma_*})$ once the indices of the inertia of $A_{\gamma} - \alpha B_{\gamma}$ change by one whenever $\gamma$ runs over $\gamma_*$ increasingly or decreasingly. 
From \cref{eq:inertia_A_B}, it follows that for $\alpha = 0^{-}$ or $0^{+}$, some diagonal entry of $C_{\gamma}(\alpha)$ must change sign as $\gamma$ increases. Therefore, there exists $\gamma = \gamma^{-} > \gamma_{*}$ and $\gamma = \gamma^{+} > \gamma_{*}$ such that $\alpha = 0^{-}$ and $0^{+}$ are eigenvalues of $(A_{\gamma}, B_{\gamma})$, respectively. 
\end{proof}

\begin{theorem}
      Suppose $\bk \neq 0$ and $n_{\ell} > 4$. The number of positive/negative eigenvalues of $(A_\gamma, B_\gamma)$ increases as $\gamma> \gamma_{*} \equiv \sqrt{\varepsilon_i}$ becomes larger;
	the new positive eigenvalue $(\mu = -1)$ and the new negative eigenvalue $(\mu = 1)$ associated with Jordan blocks of odd size appear in pairs, and they are initiated by a pair of complex conjugate eigenvalues.  
	Moreover, this kind of pair can appear as many as $\rank(\ol{\Pi_0}^{\HH}(I_3\otimes (T^{\HH}I^{(i)}T))\Pi_0)$ times.
\end{theorem}
\begin{proof}
      We first quote the following important results which have been proven above.
      \begin{enumerate}
            \item[(a)] Let $\alpha_1^{(-)}(\gamma)$ and $\alpha_1^{(+)}(\gamma)$ be the largest negative and smallest positive real eigenvalues of $(A_{\gamma}, B_{\gamma})$, for $\gamma < \gamma_{*}$. From \cref{thm:derivative_omega}, it holds that $\frac{d \alpha_1^{(-)}(\gamma)}{d \gamma} < 0$ and $\frac{d \alpha_1^{(+)}(\gamma)}{d \gamma} > 0$ generically, i.e., $\alpha_1^{(-)}$  and $\alpha_1^{(+)}$ move toward the left and the right, respectively, when $\gamma$ increases.
            
            \item[(b)] \cref{lm:eigenvalue-infty,lm:necessary-cond-of-nonreal-eigenvalue}, respectively, show that $(A_{\gamma_{*}}, B_{\gamma_{*}})$  have defective infinite eigenvalues and $\frac{d | \omega(\gamma)|}{d\gamma} < 0$, where $\ii \omega(\gamma)$ is a purely imaginary eigenvalue of $(A_{\gamma}, B_{\gamma})$, for $\gamma \to \gamma_{*}^{+}$. \cref{thm:bifurcation_ew} shows that the tangent line of the complex conjugate eigenvalue curves $\beta(\gamma) \cup \bar{\beta}(\gamma)$ is orthogonal to the real axis at some $\gamma = \gamma_1$ and bifurcate into two real eigenvalues $\alpha_{\ell}(\gamma_1^{+}) (\mu = 1)$ and $\alpha_{r}(\gamma_1^{+}) (\mu = -1)$.
      \end{enumerate}
      
From \cref{lm:nonreal-eigenvalue-burst}, we have that $\alpha = 0^{-}$ is an eigenvalue of $(A_{\gamma_2^{-}}, B_{\gamma_2^{-}})$ with $\gamma_2^{-} > \gamma_{*}$. From the continuity of eigenvalue curves and bifurcation theory, $\alpha = 0^{-}$ must have the following combination of cases listed in \cref{lm:inertia-and-eigenvalue}.
\begin{align*}
     \typ{c} \to \typ{re}  \to \typ{ro}(\mu=1) + \typ{ro}(\mu=-1).
\end{align*}

\begin{figure}[ht]
	\centering
	\includegraphics[height=2.8in]{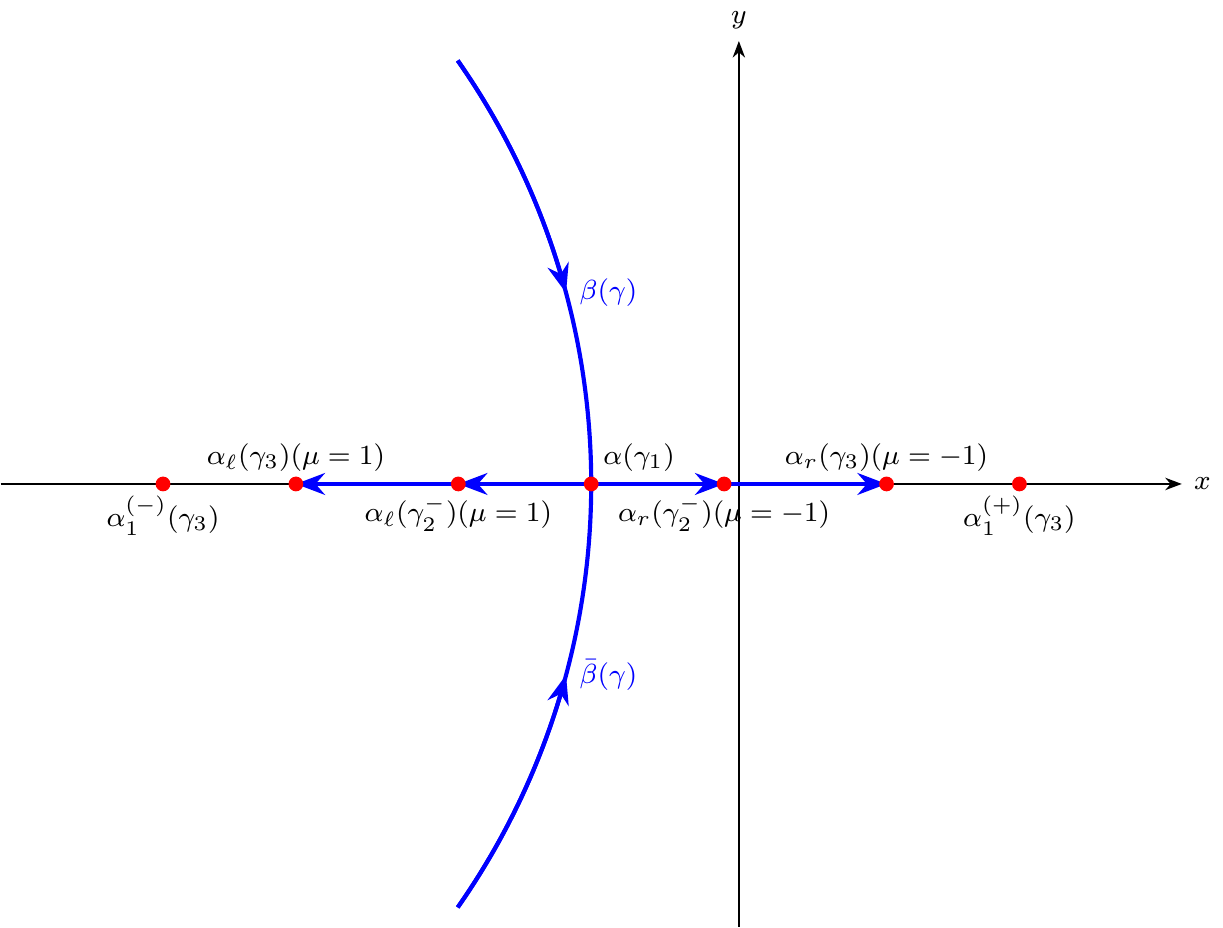}
	\caption{Scenario of bifurcation.}
	\label{fig:scenario}
\end{figure}

Scenario (see \cref{fig:scenario} for details):
\begin{enumerate}
      \item[(i)] Since $\alpha = 0^{-}$ is an eigenvalue of $(A_{\gamma_2^{-}}, B_{\gamma_2^{-}})$  and from the facts of (a) and (b), there is a complex conjugate eigenvalue pair $\{ \beta(\gamma), \bar{\beta}(\gamma) \}$ of $(A_{\gamma}, B_{\gamma})$ for $\gamma_{*} < \gamma < \gamma_1$ such that they collide at $\gamma = \gamma_1$ with $\alpha_1^{(-)}(\gamma_1) < \alpha(\gamma_1) = \beta(\gamma_1) = \bar{\beta}(\gamma_1) <0$ and bifurcate into two real eigenvalues $\alpha_{\ell}(\gamma_1^{+}) (\mu = 1)$ and $\alpha_r(\gamma_1^{+}) (\mu = -1)$  with $\alpha_{\ell}(\gamma_1^{+})$ and $\alpha_r(\gamma_1^{+})$ moving toward the left and the right, respectively.
      
      \item[(ii)] Furthermore, there exist $\gamma_1 < \gamma_2^{\pm} < \gamma_3$ such that 
      \begin{equation}
           \alpha_{\ell}(\gamma_2^{\pm}) (\mu = 1) < \alpha(\gamma_1) < \alpha_{r}(\gamma_2^{\pm}) (\mu = -1) \equiv 0^{\pm} \label{eq:Scenario_1}
      \end{equation}
      and
      \begin{equation}
           0 < \alpha_{r}(\gamma_3) < \alpha_1^{(+)}(\gamma_3) \label{eq:Scenario_2}
      \end{equation}
      which implies that a new smallest positive eigenvalue $\alpha_r(\gamma_3)$ is created at $\gamma = \gamma_3$.

	  We now show the following cases cannot happen.
      
      \item[(iii)] If there is a complex conjugate eigenvalue pair $\{ \hat{\beta}(\gamma), \bar{\hat{\beta}}(\gamma) \}$ of $(A_{\gamma}, B_{\gamma})$ that collides at $\gamma = \gamma_4>\gamma_3$ with 
      \begin{equation}
             0 < \alpha_{r}(\gamma_3) < \alpha_{r}(\gamma_4) < \hat{\alpha}_{\ell}(\gamma_4^{+})\lesssim \hat{\alpha}(\gamma_4) = \hat{\beta}(\gamma_4) = \bar{\hat{\beta}}(\gamma_4) \lesssim \hat{\alpha}_r(\gamma_4^{+}), \label{eq:Scenario_3}
      \end{equation}
      
      \item[(iv)] then, $\alpha_r(\gamma_4) (\mu = -1)$ and $\hat{\alpha}_{\ell}(\gamma_4^{+})$ will collide at $\alpha_r(\gamma_5)$ for $\gamma_4 \to \gamma_5$ increasingly, as the combination below.
      \begin{multline}
             \gamma = \gamma_5^{-}: \typ{ro} (\mu=1) + \typ{ro}(\mu=-1) 
			 \\\to \gamma = \gamma_5:  \typ{re}(\mu=-1) \to \gamma = \gamma_5^{+}: \typ{c}.
     \end{multline}
     
     \item[(v)] Then, from \cref{lm:nonreal-eigenvalue-burst}, it follows that $p_{+}(C_{\gamma_5^{-}}(0^{-})) < p_{+}(C_{\gamma_5^{+}}(0^{-}))$, which is a contradiction. 
 \end{enumerate}
     
     From \cref{eq:inertia_A_B}, \cref{lm:nonreal-eigenvalue-burst} and (ii), it follows that
     \[
          p_{-}(C_{\gamma_2^{-}}(0^{-})) < p_{-}(C_{\gamma_2^{+}}(0^{-})) 
     \]
     and
     \[
         p_{+}(C_{\gamma_2^{-}}(0^{+})) < p_{+}(C_{\gamma_2^{+}}(0^{+})).
     \]
     This implies that the new positive and negative eigenvalues $\alpha_r(\gamma_3)$ and $\alpha_{\ell}(\gamma_3)$ must be of negative $(\mu = -1)$ and positive $(\mu = 1)$ types, respectively.
     
     On the other hand, another scenario in which the pair of eigenvalue curves $\{ \beta(\gamma), \bar{\beta}(\gamma) \}$ for $\gamma_{*} < \gamma < \gamma_1$ collides at $\gamma = \gamma_1$ with
     \begin{align*}
          0 < \alpha(\gamma_1) = \beta(\gamma_1) = \bar{\beta}(\gamma_1) < \alpha_1^{(+)}(\gamma_1)
     \end{align*}
	 can also happen (see \cref{fig:4th_bifurcation} 
	 in \cref{sec:numer_reslt}). A similar discussion as in (i)-(v) above is still held by replacing $0^{-}$ by $0^{+}$ and considering all combinations symmetric to the purely imaginary axis. These two scenarios should be mutually exclusive. 

    Finally, the kind of pairs in (i) and (ii) can appear as many as 
    $$\rank(U_2) = \rank( \ol{\Pi_0}^{\HH} (I_3\otimes (T^{\HH}I^{(i)}T)) \Pi_0)$$ 
    times because the matrix $(\varepsilon_i - \gamma^2) U_2^{\HH} U_2$ in  \cref{eq:inertia_A_B} makes $B_{\gamma}$ change signs $\rank(U_2)$ times, as $\gamma \to \infty$. 
\end{proof}

\section{Numerical results} \label{sec:numer_reslt}

\begin{figure}[ht]
	\centering
	\subfigure[Schema \label{fig:schema}]{\includegraphics[height=1.75in]{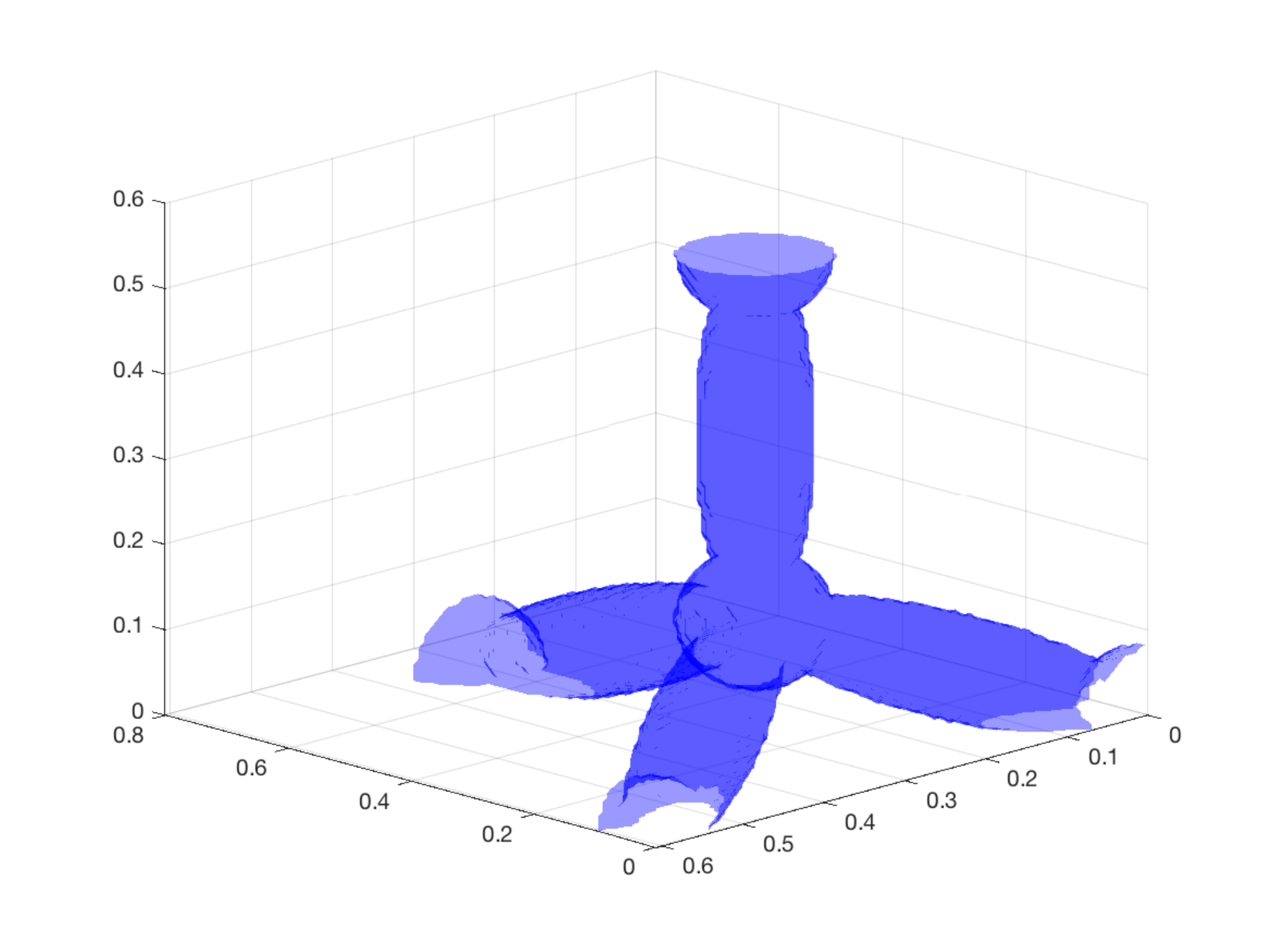}}
	\subfigure[Eigencurve-structure vs. various $\gamma$ \label{fig:eigen-structure}]{\includegraphics[height=1.75in]{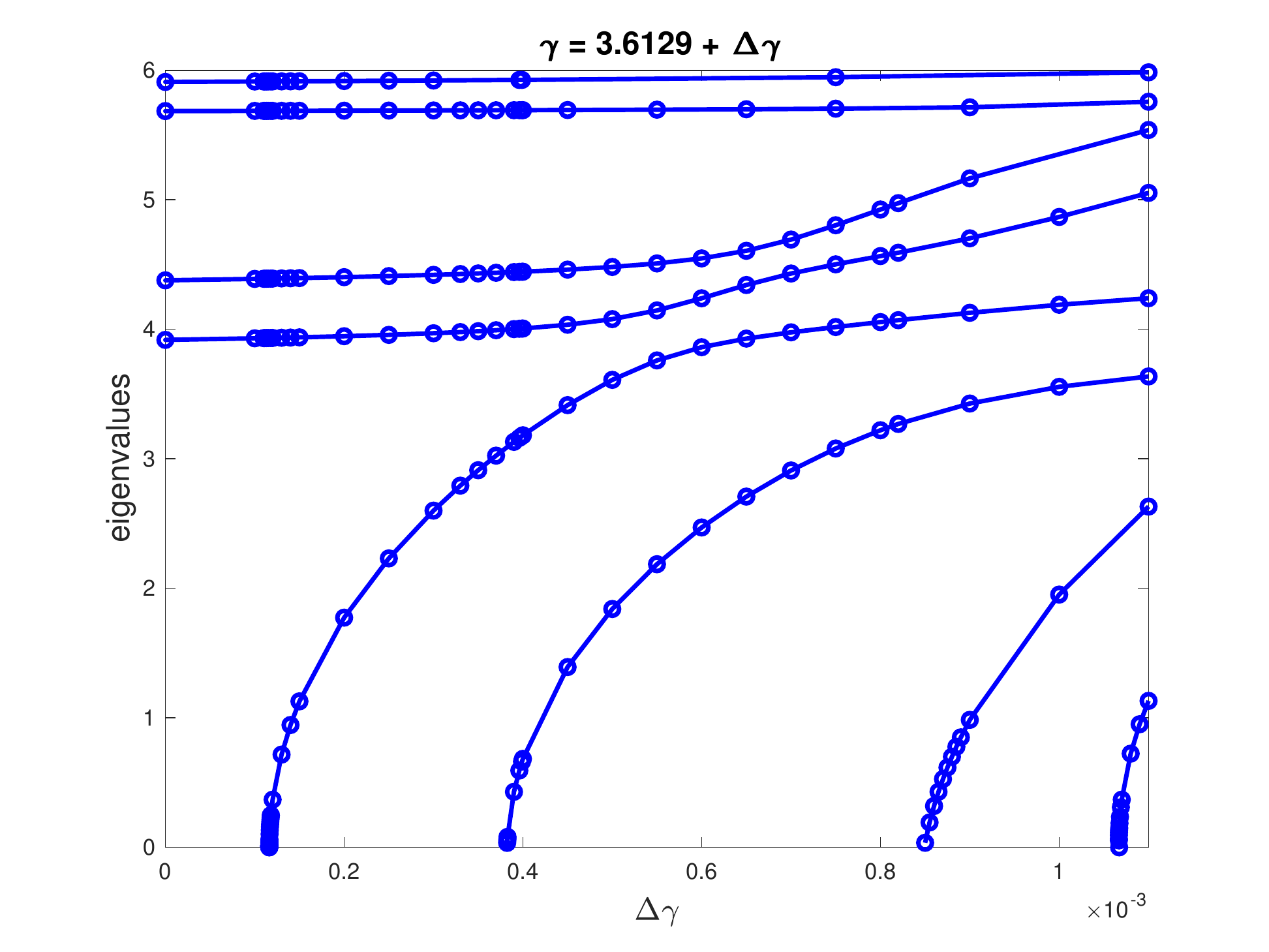}}
	\caption{A schema of 3D complex media with the FCC lattice and eigencurve-structure vs. various $\gamma$} 
	\label{fig:schema_eigencurve_gamma}
\end{figure}

To study numerical behaviors of the complex conjugate eigenvalue curves with $\gamma > \gamma_{*} \equiv \sqrt{\varepsilon_i}$,
we consider the FCC lattice \cite{hulcl:2019} which
consists of
dielectric spheres with connecting spheroids, as is shown in \cref{fig:schema}. 
 The mesh numbers $n_1$, $n_2$ and $n_3$ are taken as $n_1 = n_2 = n_3 = 96$, and the matrix dimension of $\widehat{A}_r$ in \cref{eq3.17} is 3,538,944. Here, $\varepsilon_i = 13$.
 
 As shown by the numerical results shown in \cite[Figure 4]{hulcl:2019}, there are four newly created smallest energies as $\gamma$ increases from $\sqrt{\varepsilon_i} \approx 3.6056$ to $3.61397$. The zoom-in view of the eigencurve-structure is shown in \cref{fig:eigen-structure}. 
 The results demonstrate that four newly created smallest energies are produced on the tiny increment $\Delta \gamma = 10^{-3}$ of $\gamma$. These energies emerge from lower frequencies and push the original eigenmodes to higher frequencies. These new smallest eigenvalues do not collide with the original eigenvalues so no bifurcation occurs again between these eigenvalues. 

In \cref{fig:Tow_bifurcation}, we demonstrate the local behavior of the complex conjugate eigenvalue curves which collide and bifurcate into two real eigenvalues at $\gamma = \gamma_1 \approx 3.6130162$ and $\gamma_4 \approx 3.61396676$. The results show that
the tangent lines of $\beta(\gamma)\cup \bar{\beta}(\gamma)$ and $\alpha(\gamma) \equiv \alpha_{\ell}(\gamma) \cup \alpha_r(\gamma)$ at $\gamma = \gamma_1$ and $\gamma_4$ are orthogonal to the real $x$- and imaginary $y$-axes, respectively, as the proof of \cref{thm:bifurcation_ew}. Moreover, the complex conjugate eigenvalue curves collide and bifurcate at $\alpha(\gamma_1) \approx -1.194\times 10^{-2}$ and $\alpha(\gamma_4) \approx 1.693 \times 10^{-2}$, respectively. The negative and positive eigenvalues $\alpha(\gamma_1)$ and $\alpha(\gamma_4)$ instantly move toward the right and the left, respectively, to a new positive eigenvalue $\alpha_r(\gamma_1 + \Delta \gamma_1)$ and a negative eigenvalue $\alpha_{\ell}(\gamma_4 + \Delta \gamma_4)$ along the real axis, where $\Delta \gamma_1= 6 \times 10^{-9}$ and $\Delta \gamma_4 = 2 \times 10^{-8}$.  

\begin{figure}[ht]
	\centering
	\subfigure[The first newly created smallest energy \label{fig:1st_bifurcation}]{\includegraphics[height=1.85in]{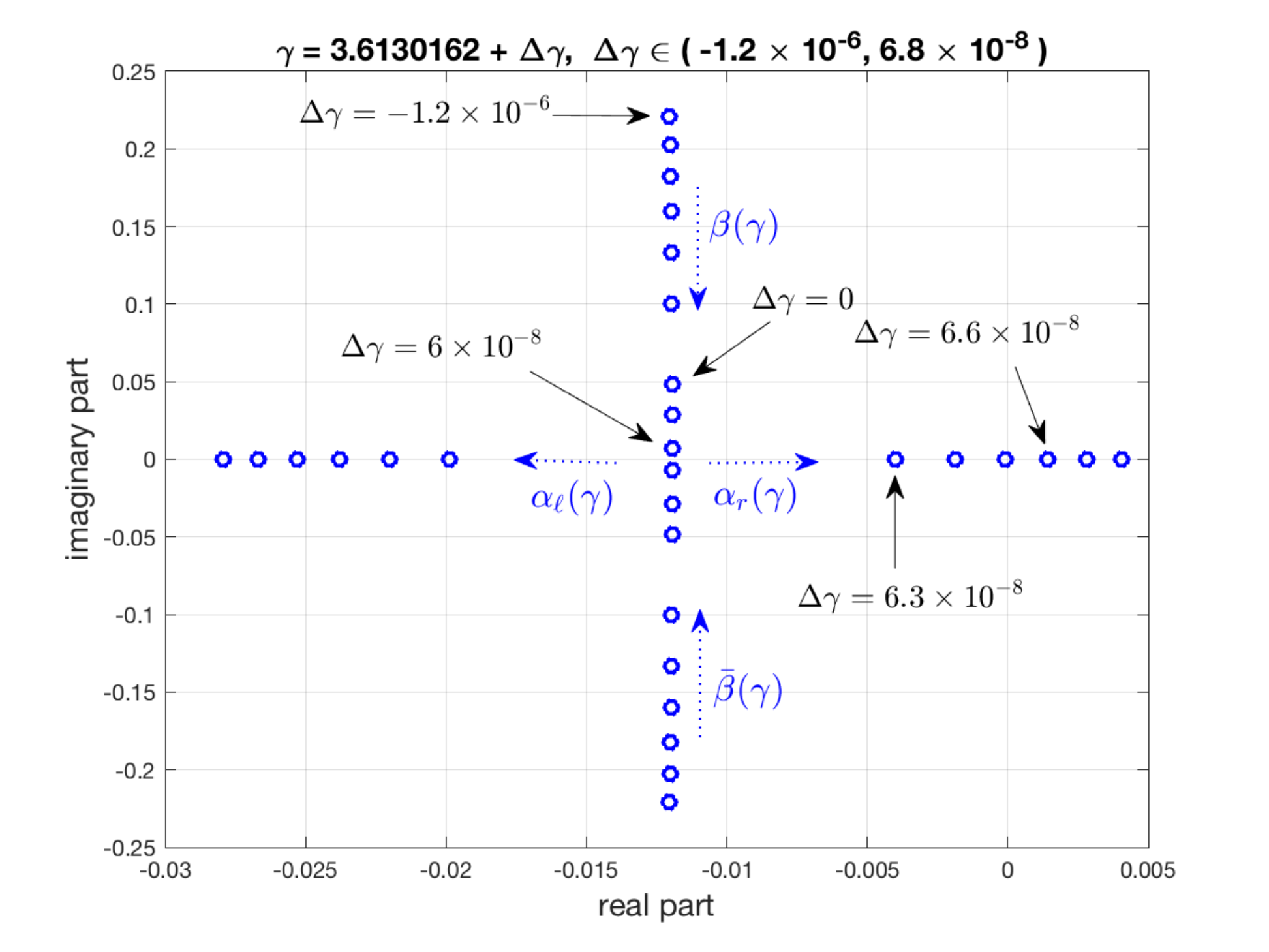}}
	\subfigure[The fourth newly created smallest energy \label{fig:4th_bifurcation}]{\includegraphics[height=1.85in]{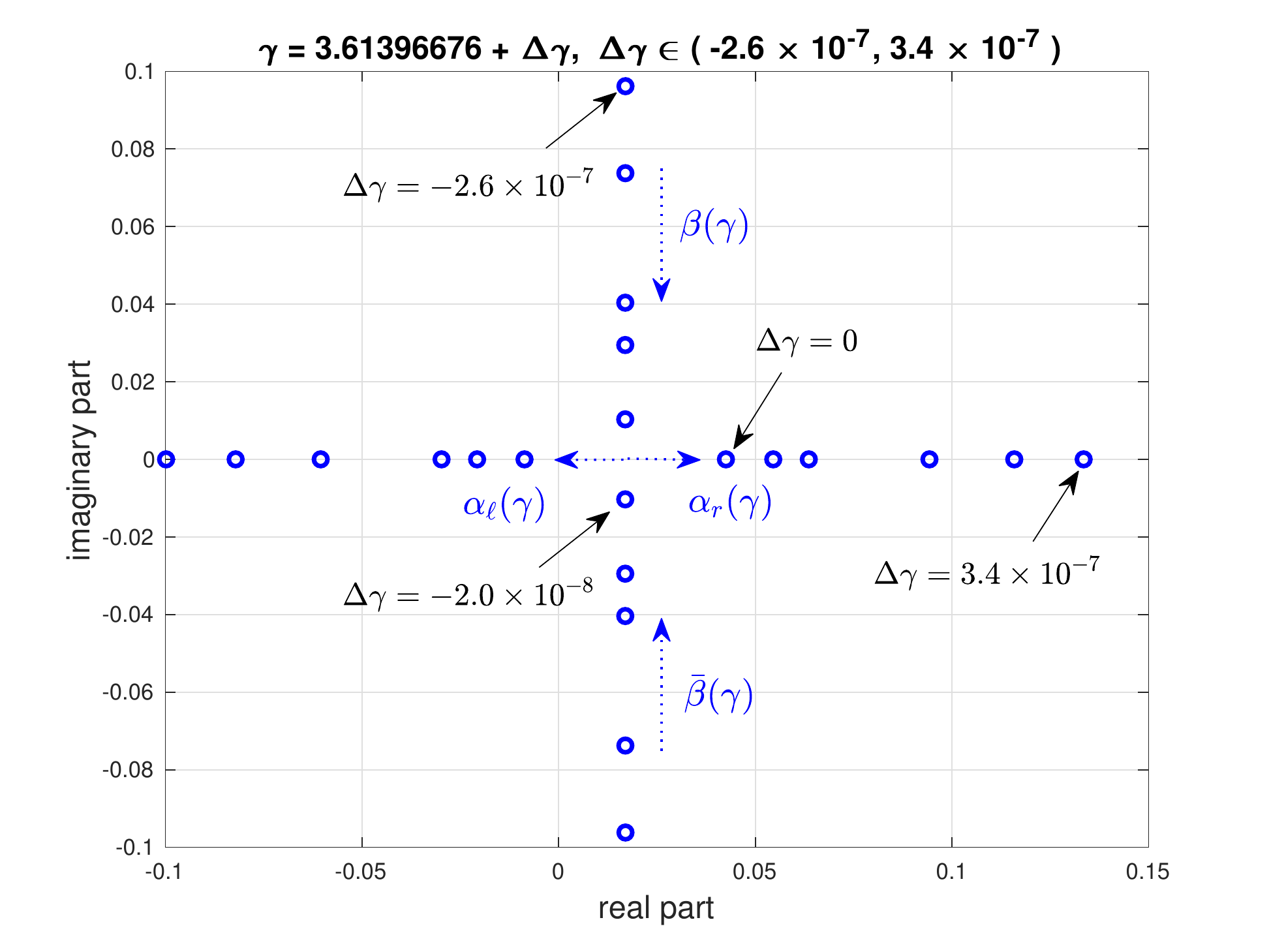}}
	\caption{The local behavior of the complex conjugate eigenvalue curves which collide and bifurcate into two real eigenvalues at $\gamma \approx 3.6130162$ and $3.61396676$.} 
	\label{fig:Tow_bifurcation}
\end{figure}

\section{Conclusions}\label{sec:conclusion}
In this paper, we prove a detailed bifurcation analysis of eigenstructures of the discrete single-curl operator in 3D Maxwell's equations with Pasteur media that depend on a chirality parameter $\gamma$ as it varies. We compensate for the theoretical difficulties and guarantee that the numerical results are valid and reliable. These results can provide an important theoretical viewpoint on numerical computations, especially regarding the support of numerical results in \cite{hulcl:2019} computed by the developed SIRA + MINRES for NFGEP. It is worth mentioning that in \cref{rk:lm:Jordan-block}, we show that the associated electric field $\be$ of the defective infinite eigenvalue is zero outside the material. This provides a very good reason to explain that the electric field corresponding to the newly created smallest energy state is almost concentrated in the material such that only a small amount of the field leak into the background material.

In the future, it would be very challenging to compute the Bloch dispersion curves corresponding to a periodic array of plasmonic nanoparticles inside a chiral background medium. 

\appendix
\section{The regularity of \texorpdfstring{$A - \omega B$}{A - omega B}}

\begin{theorem}\label{lm:regularity}
	$A-\omega B$ is always regular,
	as long as three line segments, parallel to the three mesh grid axes respectively, with end points lying on the boundary of the mesh grid are outside the medium, 
	i.e., there exist some $i_1,i_1'\in[1,n_1],i_2,i_2'\in[1,n_2],i_3,i_3'\in[1,n_3]$, such that $\calG\subset\domain_o$, where
	$\calG=\calG_1\cup\calG_2\cup\calG_3$ with
	$\calG_1=\set{\coo{i,i_2,i_3'}:  i\in\mathbb{Z}}$,
	$\calG_2=\set{\coo{i_1',i,i_3}:  i\in\mathbb{Z}}$,
	$\calG_3=\set{\coo{i_1,i_2',i}:  i\in\mathbb{Z}}$.
\end{theorem}
\begin{proof}
First, we will observe $\calS$.
To address it, we have to discuss several cases.

\newcommand\caseitem[1]{\textbf{Case #1}}
\newcommand\case[1]{\par\smallskip\caseitem{#1.}}
\case{I} 
$\Lambda_\ell,\ell=1,2,3$ are nonsingular.
The only proper $z_1,z_2,z_3$ are all zero.

First, \cref{eq:constraint:x1:1} has nontrivial solutions
if and only if there exists $i_\ell\in[1,n_\ell],\ell=1,2,3$, such that
\[
	\eta_{n_3}^{\bk\cdot\what\ba_3+i_3}\eta_{n_2n_3}^{-m_2i_2}\eta_{n_1n_3}^{-\what m_1 i_1}\eta_{n_1n_2n_3}^{m_1m_2i_1}
	=\eta_{n_2}^{\bk\cdot\what\ba_2+i_2}\eta_{n_1n_2}^{-m_1i_1}
	=\eta_{n_1}^{\bk\cdot\what\ba_1+i_1}
	,
\]
and in this case, the $\coo{i_1,i_2,i_3}$-th entry of $x_1$ is nonzero.
It is equivalent to
\begin{align*}
	\frac{i_1+\bk\cdot\what\ba_1}{n_1}-s_1
	&=\frac{i_2+\bk\cdot\what\ba_2-\frac{m_1}{n_1}i_1}{n_2}-s_2
	\\&=\frac{i_3+\bk\cdot\what\ba_3-\frac{m_2}{n_2}i_2-\frac{\what m_1}{n_1}i_1+\frac{m_1m_2}{n_1n_2}i_1}{n_3}-s_3
	=:\lambda
	\in[0,1)
	,
\end{align*}
for some $s_1,s_2,s_3\in \mathbb{Z}$.
In other words,
$x_1=I^{\calI_0}x_1$,
where
\[
	\calI_0:=\set{\coo{
			\lambda \what n_1-\kappa_1,
			\lambda \what n_2-\kappa_2,
			\lambda \what n_3-\kappa_3
		} : \lambda \what n_\ell-\kappa_\ell \in\mathbb{Z},
		0<\lambda<1
	},
\]
with
\begin{alignat*}{2}
	\what n_1&=n_1,               \quad&	\kappa_1&=\bk\cdot\ba_1, \\
	\what n_2&=n_2+m_1,           \quad&	\kappa_2&=\bk\cdot(\ba_2+\rho_1\ba_1)-m_1s_1, \\
	\what n_3&=n_3+m_2+\what m_1, \quad&	\kappa_3&=\bk\cdot(\ba_3+\rho_2\ba_2+\what \rho_1\ba_1)-m_2s_2-\what m_1s_1.
\end{alignat*}
Clearly, $x_1\ne0$ is equivalent to $(I^{\calI_0}_\sigma)^{\T} x_1\ne0$.
Moreover, if $\calI_0\ne\emptyset$, 
then it can be shown that\footnote{
	if $\coo{\lambda \what n_1-\kappa_1 ,\lambda \what n_2-\kappa_2 ,\lambda \what n_3-\kappa_3 }\in\calI_0$,
	then $\coo{\lambda' \what n_1-\kappa_1 ,\lambda' \what n_2-\kappa_2 ,\lambda' \what n_3-\kappa_3 }\in\calI_0$ for $\lambda'=\lambda+\frac{p}{\gcd(\what n_1,\what n_2,\what n_3)}$ with $p\in\mathbb{Z}$.
	On the other hand, if $\coo{\lambda \what n_1-\kappa_1 ,\lambda \what n_2-\kappa_2 ,\lambda \what n_3-\kappa_3 }\in\calI_0$ and $\coo{\lambda' \what n_1-\kappa_1 ,\lambda' \what n_2-\kappa_2 ,\lambda' \what n_3-\kappa_3 }\in\calI_0$,
	then $(\lambda-\lambda')\what n_\ell\in\mathbb{Z}$, which infers $(\lambda-\lambda')\gcd(\what n_1,\what n_2,\what n_3)\in\mathbb{Z}$ by B\'ezout's identity.
}
\[
	\calI_0=\set{\coo{\lambda \what n_1-\kappa_1 ,\lambda \what n_2-\kappa_2 ,\lambda \what n_3-\kappa_3 } :  \lambda =\lambda_0+\frac{p}{\what n_{123}}, p\in[0,\what n_{123})\cap\mathbb{Z}},
\]
with $\lambda_0\in[0,\frac{1}{\what n_{123}})$ satisfying $\lambda_0\what n_\ell-\kappa_\ell \in\mathbb{Z}$,
and $\what n_{123}=\gcd(\what n_1,\what n_2,\what n_3)$, the greatest common divisor of $\what n_1$, $\what n_2$ and $\what n_3$.
Note that $\lambda_0$ here is unique, and $\no{\calI_0}=\what n_{123}$.

Then, consider \cref{eq:constraint:x1:2}, namely, solving $I^{(o)}T\Lambda_1x_1=0$.
For ease, we mainly discuss the case that the related index sets are nonempty.
Inserting the solution of \cref{eq:constraint:x1:1} into \cref{eq:constraint:x1:2},
we have
\[
	0
	=I^{(o)}T\Lambda_1x_1
	=I^{(o)}T\Lambda_1I^{\calI_0} x_1
	=I^{(o)}TI^{\calI_0}\Lambda_1 x_1
	=I^{(o)}TI^{\calI_0}_\sigma [(I^{\calI_0}_\sigma)^{\T} \Lambda_1 x_1]
	.
\]
Recall the form of $T=[t_\ell]$ in \cref{eq:T}.
Write 
\[
	U_{\what n_{123}}=[u_{\what n_{123},p}]_{p=1,\dots,\what n_{123}},
	\qquad
	u_{\what n_{123},p} :=
	 V_{n_3}(\eta_{\what n_{123}}^p)\otimes  V_{n_2}(\eta_{\what n_{123}}^p)\otimes  V_{n_1}(\eta_{\what n_{123}}^p)
	.
\]
and $ t_0=t_{\coo{\lambda_0 \what n_1-\kappa_1 ,\lambda_0 \what n_2-\kappa_2 ,\lambda_0 \what n_3-\kappa_3 }}$.
It can be seen that 
\[
	t_{\coo{(\lambda_0+\frac{p}{\what n_{123}}) \what n_1-\kappa_1 ,(\lambda_0+\frac{p}{\what n_{123}}) \what n_2-\kappa_2 ,(\lambda_0+\frac{p}{\what n_{123}}) \what n_3-\kappa_3 }}
	= \diag(u_{\what n_{123},p})t_0
	= \diag(t_0)u_{\what n_{123},p}
	,
\]
and then $ TI^{\calI_0}_\sigma= \diag(t_0)U_{\what n_{123}}$.
Thus,
\begin{align*}
	0
	&= I^{(o)} TI^{\calI_0}_\sigma[(I^{\calI_0}_\sigma)^{\T} \Lambda_1 x_1]
	\\&= I^{(o)} \diag(t_0)U_{\what n_{123}}[(I^{\calI_0}_\sigma)^{\T} \Lambda_1 x_1]
	\\&= \diag(t_0)I^{(o)} U_{\what n_{123}}[(I^{\calI_0}_\sigma)^{\T} \Lambda_1I^{\calI_0}_\sigma][(I^{\calI_0}_\sigma)^{\T} x_1]
	.
\end{align*}
Noticing that each entry of $t_0$ is nonzero, and $(I^{\calI_0}_\sigma)^{\T} \Lambda_1I^{\calI_0}_\sigma$ is nonsingular,
$\calS=\set{0}$ is equivalent to $  I^{(o)} U_{\what n_{123}}x=0$ has only trivial solutions,
namely, $I^{(o)} U_{\what n_{123}}$ is of full column rank.

\case{II-1}
$\Lambda_\ell,\ell=2,3$ are nonsingular, but $\Lambda_1$ is singular.
By the form of $\Lambda_1$ in \cref{eq:Lambda},
\[
	\text{
		$\Lambda_1$ is singular $\Leftrightarrow$ $
		\frac{i_1+\bk\cdot\what\ba_1}{n_1}\in \mathbb{Z}
		$ for some $i_1$ $\Leftrightarrow \bk\cdot\what\ba_1=0$,
	}
\]
and for the case $i_1=n_1$.
The only proper $z_2,z_3$ are both zero.

First, \cref{eq:constraint:x1:1} has nontrivial solutions,
if and only if:
\begin{enumerate}
	\item 
		there exists $i_\ell\in[1,n_\ell],\ell=1,2,3$, $i_1\ne n_1$, such that
		\begin{align*}
			\frac{i_1+\bk\cdot\what\ba_1}{n_1}-s_1
			&=\frac{i_2+\bk\cdot\what\ba_2-\frac{m_1}{n_1}i_1}{n_2}-s_2
			\\&=\frac{i_3+\bk\cdot\what\ba_3-\frac{m_2}{n_2}i_2-\frac{\what m_1}{n_1}i_1+\frac{m_1m_2}{n_1n_2}i_1}{n_3}-s_3
			=:\lambda
			\in[0,1)
			,
		\end{align*}
		for some $s_1,s_2,s_3\in \mathbb{Z}$.
		For the case, the $\coo{i_1,i_2,i_3}$-th entry of $x_1$ is nonzero.
	\item
		there exists $i_\ell\in[1,n_\ell],\ell=2,3$, such that
		\begin{multline*}
			\frac{i_2+\bk\cdot\what\ba_2-\frac{m_1}{n_1}n_1}{n_2}-s_2
			\\=\frac{i_3+\bk\cdot\what\ba_3-\frac{m_2}{n_2}i_2-\frac{\what m_1}{n_1}n_1+\frac{m_1m_2}{n_1n_2}n_1}{n_3}-s_3
			=:\lambda
			\in[0,1)
			,
		\end{multline*}
		for some $s_1,s_2,s_3\in \mathbb{Z}$.
		For the case, the $\coo{n_1,i_2,i_3}$-th entry of $x_1$ is decided by $z_1$,
		where $z_1\in \range(I^{\calI_1}_\sigma)$, and
		\begin{align*}
			\calI_1&:=\set{\coo{
					n_1,
					\lambda \what n_2-\kappa_2,
					\lambda \what n_3-\kappa_3
				} : \lambda \what n_\ell-\kappa_\ell \in\mathbb{Z},
				0<\lambda<1
			}
			\\&=
			\set{\coo{n_1,\lambda \what n_2-\kappa_2 ,\lambda \what n_3-\kappa_3 } :  \lambda =\lambda_1+\frac{p}{\what n_{23}}, p\in[0,\what n_{23})\cap\mathbb{Z}},
		\end{align*}
		and 
		\begin{alignat*}{2}
			\what n_2&=n_2,     \quad&	\kappa_2&=\bk\cdot(\ba_2+\rho_1\ba_1)-m_1, \\
			\what n_3&=n_3+m_2, \quad&	\kappa_3&=\bk\cdot(\ba_3+\rho_2\ba_2+\what \rho_1\ba_1)-m_2s_2-\what m_1.
		\end{alignat*}
		In detail, $(I^{\calI_1}_\sigma)^{\HH}x_1=-\delta_2[(I^{\calI_1}_\sigma)^{\HH}\Lambda_2I^{\calI_1}_\sigma]^{-1}(I^{\calI_1}_\sigma)^{\HH}z_1$.
\end{enumerate}
In other words,
$x_1=I^{\calI_0\cup\calI_1}x_1$.
Moreover, if $\calI_1\ne\emptyset$, 
with $\lambda_1\in[0,\frac{1}{\what n_{23}})$ satisfying $\lambda_0\what n_\ell-\kappa_\ell \in\mathbb{Z}$,
and $\what n_{23}=\gcd(\what n_2,\what n_3)$.
Note that $\lambda_1$ here is unique, and $\no{\calI_1}=\what n_{23}$.
Clearly, $x_1\ne0$ is equivalent to $(I^{\calI_0\cup\calI_1}_\sigma)^{\T} x_1\ne0$.

Then, consider \cref{eq:proper:z,eq:constraint:x1:2}, namely, solving $I^{(o)}T\Lambda_1x_1=0,I^{(o)}Tz_1=0$.
For ease, we mainly discuss the case in which the related index sets are nonempty.
Inserting the solution of \cref{eq:constraint:x1:1} into \cref{eq:constraint:x1:2},
we have
\begin{align*}
	0
	=
	\begin{bmatrix}
		I^{(o)}T\Lambda_1x_1\\
		I^{(o)}Tz_1\\
	\end{bmatrix}
	=
	\begin{bmatrix}
		I^{(o)}T\Lambda_1I^{\calI_0} x_1\\
		I^{(o)}TI^{\calI_1} z_1\\
	\end{bmatrix}
	&=
	\begin{bmatrix}
		I^{(o)}TI^{\calI_0}_\sigma [(I^{\calI_0}_\sigma)^{\T} \Lambda_1 x_1]\\
		I^{(o)}TI^{\calI_1}_\sigma [(I^{\calI_1}_\sigma)^{\T} z_1]\\
	\end{bmatrix}
	\\&	=
	\begin{bmatrix}
		I^{(o)}TI^{\calI_0}_\sigma &\\
		&I^{(o)}TI^{\calI_1}_\sigma \\
	\end{bmatrix}
	\begin{bmatrix}
		(I^{\calI_0}_\sigma)^{\T} \Lambda_1 x_1\\
		(I^{\calI_1}_\sigma)^{\T} z_1\\
	\end{bmatrix}
	.
\end{align*}
Similarly, 
$\calS=\set{0}$ is equivalent to $ 
\begin{bmatrix}
	I^{(o)}U_{\what n_{123}}&\\
	&I^{(o)}U_{\what n_{23}}\\
\end{bmatrix}$ is of full column rank,
where
\[
	U_{\what n_{23}}=[u_{\what n_{23},p}]_{p=1,\dots,\what n_{23}},
	\qquad
	u_{\what n_{13},p} :=
	 V_{n_3}(\eta_{\what n_{23}}^p)\otimes  V_{n_2}(\eta_{\what n_{13}}^p)\otimes  V_{n_1}(1)
	.
\]

\case{II-2}
$\Lambda_\ell,\ell=1,3$ are nonsingular, but $\Lambda_2$ is singular.
By the form of $\Lambda_2$ in \cref{eq:Lambda}, we have:
\begin{enumerate}
	\item $m_1=0$:
		\[
			\text{
				$\Lambda_2$ is singular $\Leftrightarrow$ $
				\frac{i_2+\bk\cdot\what\ba_2}{n_2}\in \mathbb{Z}
				$ for some $i_2$ $\Leftrightarrow \bk\cdot\what\ba_2=0$,
			}
		\]
		and everything is similar to \caseitem{II-1}.
		$\calS=\set{0}$ is equivalent to the matrix $ 
		\begin{bmatrix}
			I^{(o)}U_{\what n_{123}}&\\
			&I^{(o)}U_{\what n_{13}}\\
		\end{bmatrix}$ is of full column rank,
		where
		\[
			U_{\what n_{13}}=[u_{\what n_{13},p}]_{p=1,\dots,\what n_{13}},
			\qquad
			u_{\what n_{13},p} :=
			 V_{n_3}(\eta_{\what n_{13}}^p)\otimes  V_{n_2}(1)\otimes  V_{n_1}(\eta_{\what n_{13}}^p)
			.
		\]
	\item $m_1\ne0$:
		\begin{align*}
			\text{ $\Lambda_2$ is singular}
			&\Leftrightarrow 
			\text{$\frac{i_2+\bk\cdot\what\ba_2-\frac{m_1}{n_1}i_1}{n_2}\in \mathbb{Z} $ for some $i_1,i_2$} 
			\\ &\Leftrightarrow \wtd n_2:=\frac{n_1}{m_1}\bk\cdot\what\ba_2\in \mathbb{Z},
		\end{align*}
and for the case $i_1=\wtd n_2, i_2=n_2$.
The only proper $z_1,z_3$ are both zero.

First, \cref{eq:constraint:x1:1} has nontrivial solutions,
if and only if:
\begin{enumerate}
	\item 
		there exists $i_\ell\in[1,n_\ell],\ell=1,2,3$, $(i_1,i_2)\ne (\wtd n_2, n_2)$, such that
		\begin{align*}
			\!\!\!\!\!\frac{i_1+\bk\cdot\what\ba_1}{n_1}-s_1
			&=\frac{i_2+\bk\cdot\what\ba_2-\frac{m_1}{n_1}i_1}{n_2}-s_2
			\\&=\frac{i_3+\bk\cdot\what\ba_3-\frac{m_2}{n_2}i_2-\frac{\what m_1}{n_1}i_1+\frac{m_1m_2}{n_1n_2}i_1}{n_3}-s_3
			=:\lambda
			\in[0,1)
			,
		\end{align*}
		for some $s_1,s_2,s_3\in \mathbb{Z}$.
		For the case, the $\coo{i_1,i_2,i_3}$-th entry of $x_1$ is nonzero.
	\item $(i_1,i_2)= (\wtd n_2, n_2)$.
		For the case, the $\coo{\wtd n_2,n_2,i_3}$-th entry of $x_1$ is decided by $z_2$,
		where $z_2\in \range(I^{\calI_2^m}_\sigma)$, and
		\begin{align*}
			\calI_2^m&:=\set{\coo{
					\wtd n_2,
					n_2,
					i_3
				}}
				.
		\end{align*}
		In detail, $(I^{\calI_2^m}_\sigma)^{\HH}x_1=-\delta_3[(I^{\calI_2^m}_\sigma)^{\HH}\Lambda_3I^{\calI_2^m}_\sigma]^{-1}(I^{\calI_2^m}_\sigma)^{\HH}z_2$.
\end{enumerate}
In other words,
$x_1=I^{\calI_0\cup\calI_2^m}x_1$.
Note that $\no{\calI_2^m}= n_3$.
Clearly, $x_1\ne0$ is equivalent to $(I^{\calI_0\cup\calI_2^m}_\sigma)^{\T} x_1\ne0$.

Then, consider \cref{eq:proper:z,eq:constraint:x1:2}, namely, solving $I^{(o)}T\Lambda_1x_1=0,I^{(o)}Tz_2=0$.
The steps proceed similarly to \caseitem{II-1}.
$\calS=\set{0}$ is equivalent to $ 
\begin{bmatrix}
	I^{(o)}U_{\what n_{123}}&\\
	&I^{(o)}U_{n_{3}}\\
\end{bmatrix}$ is of full column rank,
where
\[
	U_{n_{3}}=[u_{n_{3},p}]_{p=1,\dots,n_{3}},
	\qquad
	u_{n_{3},p} :=
	 V_{n_3}(\eta_{n_{3}}^p)\otimes  V_{n_2}(1)\otimes  V_{n_1}(1)
	.
\]
\end{enumerate}

\case{II-3}
$\Lambda_\ell,\ell=1,2$ are nonsingular, but $\Lambda_3$ is singular.
By the form of $\Lambda_3$ in \cref{eq:Lambda}, we have:
\begin{enumerate}
	\item $\what m_1=0, m_2=0$:
		\begin{align*}
			\text{$\Lambda_3$ is singular}
			&\Leftrightarrow
		\text{$ \frac{i_3+\bk\cdot\what\ba_3}{n_2}\in \mathbb{Z} $ for some $i_3$}
		\\&\Leftrightarrow \bk\cdot\what\ba_3=0,
		\end{align*}
		and everything is similar to \caseitem{II-1}.
		$\calS=\set{0}$ is equivalent to the matrix $ 
		\begin{bmatrix}
			I^{(o)}U_{\what n_{123}}&\\
			&I^{(o)}U_{\what n_{12}}\\
		\end{bmatrix}$ is of full column rank,
		where
		\[
			U_{\what n_{12}}=[u_{\what n_{12},p}]_{p=1,\dots,\what n_{12}},
			\qquad
			u_{\what n_{12},p} :=
			 V_{n_3}(1)\otimes  V_{n_2}(\eta_{\what n_{12}}^p)\otimes  V_{n_1}(\eta_{\what n_{12}}^p)
			.
		\]
	\item $\what m_1\ne0, m_2=0$:
		\begin{align*}
			\text{$\Lambda_3$ is singular}
			&\Leftrightarrow
			\text{$ \frac{i_3+\bk\cdot\what\ba_3-\frac{\what m_1}{n_1}i_1}{n_3} \in \mathbb{Z} $ for some $i_1,i_3$} 
			\\&\Leftrightarrow \wtd n_{3,1}:=\frac{n_1}{\what m_1}\bk\cdot\what\ba_3\in \mathbb{Z},
		\end{align*}
		and everything is similar to \caseitem{II-2}(2).
		$\calS=\set{0}$ is equivalent to the matrix $ 
		\begin{bmatrix}
			I^{(o)}U_{\what n_{123}}&\\
			&I^{(o)}U_{n_{2}}\\
		\end{bmatrix}$ is of full column rank,
		where
		\[
			U_{n_{2}}=[u_{n_{2},p}]_{p=1,\dots,n_{2}},
			\qquad
			u_{n_{2},p} :=
			 V_{n_3}(1)\otimes  V_{n_2}(\eta_{n_{2}}^p)\otimes  V_{n_1}(1)
			.
		\]
	\item $m_2\ne0,\what m_1=0, m_1=0$:
		\begin{align*}
			\text{ $\Lambda_3$ is singular }
			&\Leftrightarrow
			\text{$ \frac{i_3+\bk\cdot\what\ba_3-\frac{m_2}{n_2}i_2}{n_3} \in \mathbb{Z} $ for some $i_2,i_3$}
			\\&\Leftrightarrow \wtd n_{3,2}:=\frac{n_1}{\what m_1}\bk\cdot\what\ba_3\in \mathbb{Z},
		\end{align*}
		and everything is similar to \caseitem{II-2}(2).
		$\calS=\set{0}$ is equivalent to the matrix $ 
		\begin{bmatrix}
			I^{(o)}U_{\what n_{123}}&\\
			&I^{(o)}U_{\what n_{1}}\\
		\end{bmatrix}$ is of full column rank,
		where
		\[
			U_{n_{1}}=[u_{n_{1},p}]_{p=1,\dots,n_{1}},
			\qquad
			u_{n_{1},p} :=
			 V_{n_3}(1)\otimes  V_{n_2}(1)\otimes  V_{n_1}(\eta_{n_{1}}^p)
			.
		\]
	\item $m_2\ne0$, $\what m_1,m_1$ not both zero:
		\[
			\!\!\!\text{ $\Lambda_3$ is singular }
			\Leftrightarrow
			\text{$ \frac{i_3+\bk\cdot\what\ba_3-\frac{m_2}{n_2}i_2-\frac{\what m_1}{n_1}i_1+\frac{m_1m_2}{n_1n_2}i_1}{n_3} \in \mathbb{Z} $ for some $i_1,i_2,i_3$} 
				,
		\]
		and for the case where there is only one choice $(i_1,i_2, i_3)$.
		Write the single-element set as $\calI_3^m$.
The only proper $z_1,z_2$ are both zero.

First, \cref{eq:constraint:x1:1} has nontrivial solutions,
if and only if:
\begin{enumerate}
	\item 
		there exists $i_\ell\in[1,n_\ell],\ell=1,2,3$, $(i_1,i_2,i_3)\notin \calI_3^m$, such that
		\begin{align*}
			\!\!\!\!\!\frac{i_1+\bk\cdot\what\ba_1}{n_1}-s_1
			&=\frac{i_2+\bk\cdot\what\ba_2-\frac{m_1}{n_1}i_1}{n_2}-s_2
			\\&=\frac{i_3+\bk\cdot\what\ba_3-\frac{m_2}{n_2}i_2-\frac{\what m_1}{n_1}i_1+\frac{m_1m_2}{n_1n_2}i_1}{n_3}-s_3
			=:\lambda
			\in[0,1)
			,
		\end{align*}
		for some $s_1,s_2,s_3\in \mathbb{Z}$.
		For the case, the $\coo{i_1,i_2,i_3}$-th entry of $x_1$ is nonzero.
	\item $(i_1,i_2,i_3)\in \calI_3^m$.
		For the case, the $\coo{\wtd n_2,n_2,i_3}$-th entry of $x_1$ is decided by $z_3$,
		where $z_2\in \range(I^{\calI_3^m}_\sigma)$.
		In detail, 
		\[
			(I^{\calI_3^m}_\sigma)^{\HH}x_1=-\delta_3[(I^{\calI_3^m}_\sigma)^{\HH}\Lambda_3I^{\calI_3^m}_\sigma]^{-1}(I^{\calI_3^m}_\sigma)^{\HH}z_3
		.
		\]
\end{enumerate}
In other words,
$x_1=I^{\calI_0\cup\calI_3^m}x_1$.
Note that $\no{\calI_3^m}= n_1$.
Clearly, $x_1\ne0$ is equivalent to $(I^{\calI_0\cup\calI_3^m}_\sigma)^{\T} x_1\ne0$.

Then, consider \cref{eq:proper:z,eq:constraint:x1:2}, namely, solving $I^{(o)}T\Lambda_1x_1=0,I^{(o)}Tz_3=0$.
The steps proceed similarly to \caseitem{II-1}.
$\calS=\set{0}$ is equivalent to $ 
\begin{bmatrix}
	I^{(o)}U_{\what n_{123}}&\\
	&I^{(o)}U_{1}\\
\end{bmatrix}$ is of full column rank,
where
\[
	U_{1}=[u_{1,1}],
	\qquad
	u_{1,1} :=
	 V_{n_3}(1)\otimes  V_{n_2}(1)\otimes  V_{n_1}(1)
	.
\]

\end{enumerate}

\case{III-3}
\addtolength\arraycolsep{-4pt}
$\Lambda_\ell,\ell=1,2$ are singular, but $\Lambda_3$ is nonsingular.
The only proper $z_3=0$.
By the form of $\Lambda_1,\Lambda_2$ in \cref{eq:Lambda}, we have:
\begin{enumerate}
	\item $m_1=0$:
		it is similar to the combination of \caseitem{II-1} and \caseitem{II-2}(1).
		\[
			\text{
				$\Lambda_1$ is singular $\Leftrightarrow$ $
				\frac{i_1+\bk\cdot\what\ba_1}{n_1}\in \mathbb{Z}
				$ for some $i_1$ $\Leftrightarrow \bk\cdot\what\ba_1=0$,
			}
		\]
		\[
			\text{
				$\Lambda_2$ is singular $\Leftrightarrow$ $
				\frac{i_2+\bk\cdot\what\ba_2}{n_2}\in \mathbb{Z}
				$ for some $i_2$ $\Leftrightarrow \bk\cdot\what\ba_2=0$.
			}
		\]
		In detail, \[
			(I^{\calI_1}_\sigma)^{\HH}x_1=-\delta_3[(I^{\calI_1}_\sigma)^{\HH}\Lambda_3I^{\calI_1}_\sigma]^{-1}(I^{\calI_1}_\sigma)^{\HH}z_1
		,
		\]
		and \[
			(I^{\calI_2}_\sigma)^{\HH}x_1=-\delta_3[(I^{\calI_2}_\sigma)^{\HH}\Lambda_3I^{\calI_2}_\sigma]^{-1}(I^{\calI_2}_\sigma)^{\HH}z_2.
		\]
				This forces 
		$(I^{\calI_1\cap\calI_2}_\sigma)^{\HH}z_1=(I^{\calI_1\cap\calI_2}_\sigma)^{\HH}z_2$.

		Then, consider \cref{eq:proper:z,eq:constraint:x1:2}, namely, solving $I^{(o)}T\Lambda_1x_1=0,I^{(o)}Tz_1=0,I^{(o)}Tz_2=0$.
		For ease, we mainly discuss the case in which the related index sets are nonempty.
		Inserting the solution of \cref{eq:constraint:x1:1} into \cref{eq:constraint:x1:2},
		we have
		\[
			0
			=
			\begin{bmatrix}
				I^{(o)}T\Lambda_1x_1\\
				I^{(o)}Tz_1\\
				I^{(o)}Tz_2\\
				(I^{\calI_1\cap\calI_2}_\sigma)^{\HH}(z_1-z_2)\\
			\end{bmatrix}
			=
			\begin{bmatrix}
				I^{(o)}TI^{\calI_0}_\sigma [(I^{\calI_0}_\sigma)^{\T} \Lambda_1 x_1]\\
				I^{(o)}TI^{\calI_1}_\sigma [(I^{\calI_1}_\sigma)^{\T} z_1]\\
				I^{(o)}TI^{\calI_2}_\sigma [(I^{\calI_2}_\sigma)^{\T} z_2]\\
				(I^{\calI_1\cap\calI_2}_\sigma)^{\HH}(z_1-z_2)\\
			\end{bmatrix}
			=
			\wtd T
			\begin{bmatrix}
				(I^{\calI_0}_\sigma)^{\T} \Lambda_1 x_1\\
				(I^{\calI_1\setminus\calI_2}_\sigma)^{\T} z_1\\
				(I^{\calI_1\cap\calI_2}_\sigma)^{\T} z_1\\
				(I^{\calI_2\setminus\calI_1}_\sigma)^{\T} z_2\\
				(I^{\calI_1\cap\calI_2}_\sigma)^{\T} z_2\\
			\end{bmatrix}
			,
		\]
		where $
		\wtd T=\!
			\begin{bmatrix}
				I^{(o)}TI^{\calI_0}_\sigma &&&&\\
				&I^{(o)}TI^{\calI_1\setminus\calI_2}_\sigma &I^{(o)}TI^{\calI_1\cap\calI_2}_\sigma &&\\
				&&&I^{(o)}TI^{\calI_2\setminus\calI_1}_\sigma &I^{(o)}TI^{\calI_1\cap\calI_2}_\sigma \\
				&&I &&-I \\
			\end{bmatrix}$.
		This equation has only trivial solutions, as long as $
		\begin{bmatrix}
			I^{(o)}TI^{\calI_0}_\sigma &&\\
			&I^{(o)}TI^{\calI_1}_\sigma &\\
			&&I^{(o)}TI^{\calI_2}_\sigma\\
		\end{bmatrix}
		$
		is of full column rank.
		Thus
		$\calS=\set{0}$, as long as $
		\begin{bmatrix}
			I^{(o)}U_{\what n_{123}}&&\\
			&I^{(o)}U_{\what n_{23}}&\\
			&&I^{(o)}U_{\what n_{13}}\\
		\end{bmatrix}$ is of full column rank.

	\item $m_1\ne0$:
		it is similar to \caseitem{III-3}(1), considering the combination of \caseitem{II-1} and \caseitem{II-2}(2).
		Thus $\calS=\set{0}$, as long as $
		\!\begin{bmatrix}
			I^{(o)}U_{\what n_{123}}&&\\
			&I^{(o)}U_{\what n_{23}}&\\
			&&I^{(o)}U_{\what n_{3}}\\
		\end{bmatrix}$ is of full column rank.
\end{enumerate}

\case{III-2}
$\Lambda_\ell,\ell=1,3$ are singular, but $\Lambda_2$ is nonsingular.
It is similar to \caseitem{III-1}, considering the combination of \caseitem{II-1} and \caseitem{II-3}, 
\case{III-1}.
$\Lambda_\ell,\ell=2,3$ are singular, but $\Lambda_1$ is nonsingular.
It is similar to \caseitem{III-1}, considering the combination of \caseitem{II-2} and \caseitem{II-3}.

\case{IV}
$\Lambda_\ell,\ell=1,2,3$ are all singular.
\begin{enumerate}
	\item $m_1=0,\what m_1=0,m_2=0$:
		it is similar to \caseitem{III-1}(1), considering the combination of \caseitem{II-1}, \caseitem{II-2} and \caseitem{II-3}.
		Since $\Lambda_q=\Lambda_1^{\HH}\Lambda_1+\Lambda_2^{\HH}\Lambda_2+\Lambda_3^{\HH}\Lambda_3\succ0$, 
		we know $\calI_1\cap\calI_2\cap\calI_3=\emptyset$.
		Note that 
		$(I^{\calI_1\cap\calI_2}_\sigma)^{\HH}z_1=(I^{\calI_1\cap\calI_2}_\sigma)^{\HH}z_2$,
		$(I^{\calI_1\cap\calI_3}_\sigma)^{\HH}z_1=(I^{\calI_1\cap\calI_3}_\sigma)^{\HH}z_3$, and
		$(I^{\calI_3\cap\calI_2}_\sigma)^{\HH}z_3=(I^{\calI_3\cap\calI_2}_\sigma)^{\HH}z_2$.
		Thus,
		$\calS=\set{0}$, as long as $
		\begin{bmatrix}
			I^{(o)}U_{\what n_{123}}&&&\\
			&I^{(o)}U_{\what n_{23}}&&\\
			&&I^{(o)}U_{\what n_{13}}&\\
			&&&I^{(o)}U_{\what n_{12}}\\
		\end{bmatrix}$ is of full column rank.
	\item other cases: everything is similar.
\end{enumerate}
\addtolength\arraycolsep{4pt}

To summarize, $\calS=\set{0}$, as long as all the matrices below are of full column rank:
\[
	I^{(o)}U_{\what n_{123}},
	I^{(o)}U_{\what n_{12}},
	I^{(o)}U_{\what n_{23}},
	I^{(o)}U_{\what n_{13}},
	I^{(o)}U_{n_{1}},
	I^{(o)}U_{n_{2}},
	I^{(o)}U_{n_{3}},
	I^{(o)}U_{1}.
\]

Under the condition,
\begin{enumerate}
	\item $I^{(o)}U_{1}$ is of full rank because there is only one column, and each entry is $1$.
	\item if $\calG_1=\set{\coo{i_1,i_2,i_3}: i_1\in \mathbb{Z}}\subset \domain_o$:\\
		then $(I^\calG_\sigma)^{\T}I^{(o)}U_{\what n_{123}}=\eta_{\what n_{123}}^{i_2+i_3-2} V_{n_1\times \what n_{123}}(\eta_{\what n_{123}})$.
		Since $\abs{\eta_{\what n_{123}}}=1$ and the upper square block of $ V_{n_1\times \what n_{123}}(\eta_{\what n_{123}})$ is $ V_{\what n_{123}\times \what n_{123}}(\eta_{\what n_{123}})$, the DFT matrix of size $\what n_{123}$ that is nonsingular, we know $(I^\calG_\sigma)^{\T}I^{(o)}U_{\what n_{123}}$ is of full column rank, and so is $I^{(o)}U_{\what n_{123}}$.
		Similarly, $ I^{(o)}U_{\what n_{12}}, I^{(o)}U_{\what n_{13}}, I^{(o)}U_{n_{1}}$ are of full column rank.
	\item if $\calG_2=\set{\coo{i_1,i_2,i_3}: i_2\in \mathbb{Z}}\subset \domain_o$:\\
		then similarly
		$I^{(o)}U_{\what n_{123}}, I^{(o)}U_{\what n_{12}}, I^{(o)}U_{\what n_{23}}, I^{(o)}U_{n_{2}}$ are of full column rank.
	\item if $\calG_3=\set{\coo{i_1,i_2,i_3}: i_3\in \mathbb{Z}}\subset \domain_o$:\\
		then similarly
		$I^{(o)}U_{\what n_{123}}, I^{(o)}U_{\what n_{13}}, I^{(o)}U_{\what n_{23}}, I^{(o)}U_{n_{3}}$ are of full column rank.
\end{enumerate}
As a result, we have the lemma.
\end{proof}

{\small
\bibliographystyle{siam}
\bibliography{eigpasteur}
}
\end{document}